\documentclass[11pt]{article}
\usepackage[all,2cell]{xy} \UseAllTwocells \SilentMatrices
\usepackage{latexsym,amsfonts,amssymb}
\usepackage{amsmath,amsthm,amscd}
\usepackage{hyperref,psfrag}
\usepackage{graphicx}
\usepackage{epsfig}
\usepackage{cite, tocvsec2}
\usepackage{amsfonts, amstext, mathrsfs}
\usepackage{amsopn,amscd,xspace,enumerate,multicol}

\usepackage{a4wide}

\normalfont\upshape

\usepackage{fancyheadings}
\newcommand\lie[1]{\mathfrak{#1}}
\newcommand{\lb}{\lie{b}}
\newcommand{\lu}{\lie{u}}
\newcommand{\n}{\lie{n}}
\newcommand{\p}{\lie{p}}
\newcommand{\g}{\lie{g}}
\newcommand{\s}{\lie{sl}}
\newcommand{\h}{\lie{h}}

\newcommand{\C}{\mathbb{C}}

\newcommand{\PS}{\mathbb{P}}

\newcommand{\Q}{\mathbb{Q}}
\newcommand{\Z}{\mathbb{Z}}
\newcommand{\kf}{\mathbb{K}}

\newcommand{\N}{\mathbb{N}}
\newcommand{\Sch}{\mathbb{S}}
\newcommand{\Nt}{{ \widetilde{\mathcal{N}} }}
\newcommand{\Ox}{\mathcal{O}}

\newcommand{\F}{\mathcal{F}}

\newcommand\op[1]{{\rm{#1}}}
\newcommand\mH{\mathrm{H}}
\newcommand\lra{\longrightarrow}

\newcommand\pr{\mathrm{pr}}

\newcommand{\Ul}{{\mathsf U}}
\newcommand{\ul}{{\mathsf u}}
\newcommand{\Zl}{{\mathsf Z }}
\newcommand{\zl}{{\mathsf z }}

\newcommand{\Pg}{\mathcal{P}}
\newcommand{\Qg}{\mathcal{Q}}
\newcommand{\Cg}{\mathcal{C}}

\theoremstyle{theorem}
\newtheorem{theorem}{Theorem}[section]
\newtheorem{corollary}[theorem]{Corollary}
\newtheorem{conjecture}[theorem]{Conjecture}
\newtheorem{lemma}[theorem]{Lemma}
\newtheorem{proposition}[theorem]{Proposition}

\theoremstyle{definition}
\newtheorem{definition}[theorem]{Definition}
\newtheorem{example}[theorem]{Example}

\numberwithin{equation}{section}

\title{The center of small quantum groups II: singular blocks}
\author{Anna Lachowska, You Qi}
\date{\today}

\begin{document}

\maketitle

\begin{abstract}
We generalize to the case of singular blocks the result in \cite{BeLa} that describes the center of the principal block of a small quantum group in terms of sheaf cohomology over the Springer resolution. Then using the method developed in \cite{LQ1}, we present a linear algebro-geometric approach to compute the dimensions of the singular blocks and of the entire center of the small quantum group associated with a complex semisimple Lie algebra. A conjectural formula for the dimension of the center of the small quantum group at an $l$th root of unity is formulated in type A.

\vspace{0.2in}

\end{abstract}

\section{Introduction}
The present paper is a continuation of our project started in \cite{LQ1}\footnote{We refer the reader to this paper for more background material and motivations.} to find a combinatorial description of the center of the small quantum group associated to a simple Lie algebra $\g$ at a primitive $l$th root of unity. While in \cite{LQ1} 
we  were concerned with the principal block of the center, here our main object of consideration 
will be the singular blocks. 

The small quantum group $\ul_q(\g)$, defined by Lusztig in \cite{LusfdHopf}, is a finite-dimensional Hopf 
algebra associated to a semisimple Lie algebra $\g$ and an (odd) $l$th root of unity $q$. It occurs as a subquotient of the big quantum group, as defined by Lusztig, the latter generated by the divided powers of the usual Chevalley generators. 

Let $\Pg$ denote the weight lattice and $\Qg$ the root lattice for $\g$, and $W$ be the (finite) Weyl group of $\g$. The set of roots of $\g$ will be written as $\Phi$, which contains the positive roots $\Phi^+$ once a set of simple roots $\Pi$ is fixed. Let $\Pg^+$ denote the set of dominant integral weights. The set of \emph{$l$-restricted dominant integral weights}
 $$\Pg_l := \{ \mu \in \Pg^+ | \langle \mu, \check{\alpha} \rangle \leq l-1, \;\; \forall \alpha \in \Pi \}$$
  is in bijection with the set of the isomorphism classes of finite dimensional simple $\ul_q(\g)$-modules \cite{LusfdHopf}. 
Let $\rho = \frac{1}{2} \sum_{\alpha \in \Phi^+} \alpha$ be the half sum of the positive roots and consider the $\rho$-shifted   
``dot'' action of the extended affine Weyl group $\widetilde{W}_\Pg = W \ltimes l\Pg$ on $\Pg$.  Then $\ul_q(\g)$ 
decomposes as a direct sum of two-sided ideals (blocks) parametrized by the orbits of the shifted extended affine Weyl group action in $\Pg_l$.  

Let  $\mu \in \Pg_l$ be an $l$-restricted dominant integral weight and consider the only block  
in $\ul_q(\g)$ that contains the simple module of the highest weight $\mu$ with respect to the left-multiplication $\ul_q(\g)$-action. 
If the stabilizer of $\mu$ with respect to the dot action of $\widetilde{W}_\Pg$ is trivial, then the 
center of the considered block is isomorphic as an algebra to the center $\zl^0$ of the principal block.  In \cite{LQ1} we calculated the structure of  
$\zl^{0}$ for $\g = \mathfrak{sl}_3$ and $\s_4$, and proposed a conjecture that in type $A_n$, as a bigraded vector space,  
$\zl^{0}$ is isomorphic to Haiman's diagonal coinvariant algebra \cite{Hai} for the group $S_{n+1}$.  

Now suppose that $\mu \in \Pg_l$ has a nontrivial stabilizer in $\widetilde{W}_\Pg$ with respect to the dot action. 
 Let 
 $$\bar{\Cg} = \{\mu \in \Pg : 0 \leq \langle \mu + \rho, \check{\alpha} \rangle \leq l, \;\; \forall \alpha \in \Phi^+\}$$
denote the closed fundamental alcove, and  $\bar{\Cg}_{\rm sing}$ denote the union of the walls of $\bar{\Cg}$ containing the weight $-\rho$, so that 
 $$\bar{\Cg}_{\rm sing} =  \{ \nu \in \Pg : \;  \exists \, \alpha \in \Phi^+ : \; \langle \nu + \rho, \check{\alpha} \rangle = 0 \} \cap \bar{\Cg}.$$
Choose a weight  $\lambda \in \widetilde{W}_\Pg \cdot \mu \,  \cap \, \bar{\Cg}_{\rm sing}$.  
Then $\lambda$ is a singular weight with respect to the dot action of the finite Weyl 
group $W$, and it lies on the walls of the closed fundamental alcove. 
We will denote by $\ul^\lambda$ the block of $\ul_q(\g)$ containing the simple module of highest weight $\mu \in 
\widetilde{W}_\Pg \cdot \lambda$  as a composition factor.  Note that in general the choice of $\lambda$ is not unique. 
 Our goal is to develop a method to investigate the bigraded structure of the center 
 $\zl^{\lambda}$ of  $\ul^{\lambda}$. 

The paper is organized as follows. We start Section \ref{sec-parabolic-Springer} with a generalization to the singular block case of two crucial results stated in \cite{ABG, BeLa}.  We establish 
the equivalence between the singular block of the (big and the small) quantum groups, on one side, and 
certain categories of coherent sheaves on the parabolic Springer resolution (Theorem \ref{BeLa_cat} and Corollary \ref{BeLa_sing}). Our approach is based on \cite{BaKr, BaKrSing}.  Then, following the method of \cite{LQ1},  we provide an explicit description of the equivariant structure of the push-forward sheaves that appear in Corollary \ref{BeLa_sing}. The result is formulated in Theorem \ref{thm-equ-structure-V1}  and Corollary \ref{cor-wedge-P-structure}. This allows us to use the relative Lie algebra 
 cohomology approach, given by Lemma \ref{lem-reduce-to-flag}, Theorem \ref{Bott_rel_lie} and Proposition \ref{rel_lie_BGG}, to compute the structure of the center of a singular block.  Finally, Corollary \ref{cor-sl2-subalgebra}
 describes the structure of the largest known subalgebra 
 in a (singular) block of the center. It is generated by the Harish-Chandra part of the center 
 (which is isomorphic to the parabolic coinvariant algebra) and an element that corresponds in the geometric setting 
 to the Poisson bi-vector field on the parabolic Springer resolution. This statement is parallel to the case of the principal block 
 $\zl^0$ stated in \cite{LQ1}.

 Section \ref{sec-examples} contains results on the structure of the center of the singular blocks in some particular cases. Theorem \ref{thm-diagonal-ones} 
states that all zeroth sheaf cohomologies appearing in the description of the center of any block are spanned by the powers of  the Poisson bivector field.     
 Corollary \ref{cor-algebra-structure} provides the structure of the singular block of the center $\zl^{\lambda}$ for $\lambda$ corresponding to the parabolic subgroup $P = P_\lambda$ such that $G/P = \PS^n$.  These results allow us to describe every singular block in the center of the small quantum group $\ul_q(\mathfrak{sl}_3)$. Subsection 3.3 contains a computation of the dimensions of the singular blocks in case $\g = \s_4$. Together with the results for the dimension of the regular block $\zl^0$ for $\ul_q(\mathfrak{sl}_3)$ and $\ul_q(\s_4)$ obtained in \cite{LQ1}, this leads to a dimension formula for the whole center 
 of the small quantum group in these two cases, which is polynomial in $l$.  Based on this evidence,  we conjecture that the dimension of the center of $\ul_q(\s_n)$ at the $l$-th root of unity, with some mild conditions on $l$, is given by the Catalan number $c_{(n+1)l-n, n}  =  \frac{1}{(n+1)l} {{(n+1)l}\choose{n}}$ (see Conjecture \ref{conj-main}).

Appendix \ref{sub_grassmannian} contains a detailed computation for a particular singular block of $\ul_q(\s_4)$.
In Appendix \ref{app-b2} we present a computation of the dimension of the center for the regular block in type 
$B_2$. This example shows that 
our conjecture for the center of the regular block of small quantum group in type $A$ does not generalize directly 
to other types.

\paragraph{Acknowledgments.} The authors would like to thank Sergey Arkhipov, Roman Bezrukavni-kov, Peng Shan, Chenyang Xu and Zhiwei Yun for many helpful discussions. Special thanks are due to Igor Frenkel for suggesting an inspiring formula for the dimension of the center for small quantum groups in type $A$, and to Bryan Ford for his valuable help with coding the computations in type $A_3$. 

Y.~Q.~is partially supported by the National Science Foundation DMS-1763328.

\section{Parabolic Springer resolution and the tangent bundle}\label{sec-parabolic-Springer}
We fix the following notations. Let $G$ be a complex (almost) simple and simply-connected algebraic group, and $P$ be a parabolic subgroup.  Define $N_P$ to be the unipotent subgroup in $P$, and $L_P$ to be the Levi subgroup, so that $P=L_PN_P$. Also fix a Cartan and Borel subgroup such that $H\subset B \subset  P$. The opposite unipotent group to $N_P$ will be denoted by $U_P$. The corresponding Lie algebras of the respective groups will be denoted by the corresponding Gothic letters.
\begin{equation}
\begin{array}{c}\label{eqn-lie-algebras}
\g: =\mathrm{Lie}(G), \quad \p:=\mathrm{Lie}(P), \quad \h: =\mathrm{Lie}(H), \quad \lb:=\mathrm{Lie}(B),\\
\mathfrak{l}_P: =\mathrm{Lie}(L_P), \quad \ \ \n_P:=\mathrm{Lie}(N_P), \quad \ \  \lu_P:=\mathrm{Lie}(U_P). 
\end{array} 
\end{equation}
The Weyl group of $G$ will be written as $W$, and similarly $W_P$ for that of the Levi $L_P$.

Let $X:=G/P$ be the associated partial flag variety, and set
\begin{equation}\label{eqn-Springer-variety}
\Nt_P:=T^*X
\end{equation} 
to be the cotangent bundle of $X$. It is a holomorphic symplectic variety with a canonical symplectic form $\omega$. The action of $G$ on $\Nt_P$ is well-known to be Hamiltonian (see, for instance, \cite{CG} for more details).

Let $\kf$ be a base field (algebraically closed of characteristic zero), $E$ be a $\kf$-vector space and $\Phi \in E$ the root system of $\g$, which contains the  \emph{positive roots} $\Phi^+$ once a set of simple roots 
$\Pi = \{ \alpha_i\}_{i \in I}$ is fixed. Let $\check{\alpha} \in \check{\Phi} \subset E^*$ denote the coroot corresponding to a root $\alpha \in \Phi$, so that $\langle \check{\alpha}_i, \alpha_j \rangle = a_{ij}$ is the Cartan matrix. 

The Weyl group $W$ acts naturally on the \emph{root lattice} $\Qg$, \emph{weight lattice} $\Pg$ and the \emph{coweight lattice} $\check{\Pg}$:
$$\Qg = \bigoplus_{i \in I} \Z \alpha_i\quad \quad \Pg = \{ \mu \in E^* : \langle \mu, \check{\alpha}_i \rangle \in \Z , \forall i \in I\}, \quad \quad \check{\Pg}:= \op{Hom}(\Qg, \Z) \subset E. $$ 
There is a unique $W$-invariant inner 
product $(-,-):\Qg \times \Qg \to \Q$ such that $(\alpha_i, \alpha_i) = 2 d_i $, for all $ i \in I$, where $d_i \in \{ 1,2,3\}$, and the 
matrix $(d_i a_{ij})$ is symmetric. Recall that a weight $\lambda$ is called \emph{dominant} if $\langle \lambda, \check{\alpha}_i \rangle\geq 0$. To such a weight is attached the finite-dimensional simple $G$-representation $L_\lambda$.

Let $\Ul_v(\g)$,  or just $\Ul_v$ if $\g$ is clear, be the Drinfeld-Jimbo quantum group over $\Q(v)$, where $v$ is the formal parameter. It is generated by the Chevalley generators $\{ F_i, E_i, K_i^{\pm 1} \}_{i \in I}$
and the well-known relations (see, e.g., \cite{Lus4, dCKac}).

Throughout, we will fix $l$ to be a prime number  
that is greater than the Coxeter number of $\g$.\footnote{The restriction to the prime $l$ is determined by the fact that the main results in \cite{BaKr,BaKrSing} are obtained by reduction to the modular case. We believe that most of our results will hold for any odd $l$ greater than the Coxeter number and coprime to $3$ if $\g$ has a component of type $G_2$. The proof should be based then on the singular analog of the equivalence of categories in \cite{ABG} obtained in \cite{ACR}.}  Our condition implies in particular that  $l$ is coprime to the determinant of the Cartan matrix of $\g$, which facilitates the description of the block structure 
of the small quantum group (see Proposition \ref{prop_number_blocks}) and allows us to formulate Conjecture \ref{conj-main}. 

\subsection{Singular blocks of small quantum groups}\label{sec-small-q-group}
Below we derive a geometric description of the singular blocks of the center, analogous to the result in \cite[Section 2.4]{BeLa}.

\paragraph{The small quantum group.} 
 Let $\Ul_v^{\op{res}}$ be the $\Q[v, v^{-1}]$-subalgebra 
of $\Ul_v$ generated by the divided powers 
$E_i^{(n)} = \frac{E_i^n}{[n]_{d_i}!}$, $F_i^{(n)} = \frac{F_i^n}{[n]_{d_i}!}$, where $i \in I$, 
$n \geq 1$. 
 Here $[n]_{d}! = \prod_{k=1}^n \frac{q^{dk} - q^{-dk}}{q - q^{-1}}$, and $d_i = \frac{1}{2} (\alpha_i, \alpha_i)$.
Lusztig obtained a description of $\Ul_v^{\op{res}}$ in terms of generators and relations, and, in particular, he defined in 
$\Ul_v^{\op{res}}$ the elements $E_\alpha^{(n)}, F_\alpha^{(n)}$ corresponding to all positive roots $\alpha \in \Phi^+$ and $n \geq 1$ (see \cite{LusfdHopf, Lusrootsof1}). 
Let $\Ul_q^{\op{res}} = \Ul_v^{\op{ res}} \otimes_{\Q[v, v^{-1}]} \Q(q)$, where $q$ is a primitive $l$-th root of unity. 
Then we have 
$E_\alpha^l =0$, $F_\alpha^l =0$ for all $\alpha \in  \Phi^+$, $K_i^l$ are central and $K_i^{2l} =1$ for $i \in I$ in 
$\Ul_q^{\op{res}}$. The {\it adjoint small quantum group of type {\bf 1}}, denoted $\ul \equiv \ul_q(\g)$, is the $\Q(q)$- subquotient of $\Ul_q^{\op{res}}$, 
generated by $\{ E_i, F_i, K_i^{\pm 1}  \}_{i \in I}$ and $\frac{K_i - K_i^{-1}}{q^{d_i} - q^{-d_i}} , i \in I$, and  factorized over the ideal 
generated by $\{ K_i^l -1\}_{i \in I}$. If $\g$ is simply laced, $\ul$ has the dimension $l^{\mathrm{dim}\g}$ over $\Q(q)$.  

Let $\lambda\in \Pg$ be an integral weight in the closed fundamental alcove 
\begin{equation}
\bar{\Cg} = \{\mu \in \Pg : 0 \leq \langle \mu + \rho, \check{\alpha} \rangle \leq l, \;\; \forall \alpha \in \Phi^+\},
\end{equation}
and consider the union of the walls of $\bar{\Cg}$ containing the weight $-\rho$:  
\begin{equation}
\bar{\Cg}_{\rm sing}: = \bar{\Cg} \; \cap \; \{ \nu \in \Pg : \;  \exists \alpha \in \Phi^+ : \; \langle \nu + \rho, \check{\alpha} \rangle = 0 \} .
\end{equation}
In this paper we will call the weights in $\bar{\Cg}_{\rm sing}$ the \emph{singular weights}. 
If $\lambda \in \bar{\Cg}_{\rm sing} $, then $\lambda$ has a nontrivial stabilizer group with respect to the dot action of the finite Weyl group $W$. Let $P(=P_\lambda)$ be the parabolic subgroup whose parabolic roots equal to the singular roots of $\lambda$. Then the Weyl group $W_P$ for $L_P$ can be identified with the stabilizer group  with respect to the dot action of $\lambda$ in $W$.

Let $\mathrm{Rep}(\ul)$ be the category of finite-dimensional $\ul_q(\g)$ modules. Then 
for a fixed $\lambda   \in \bar{\Cg}_{\rm sing}$, 
 the block $\mathrm{Rep}(\ul)^{\lambda}$ is the indecomposable abelian subcategory of the category 
of finite dimensional $\ul_q(\g)$-modules such that $M \in  \mathrm{Rep}(\ul)^{\lambda}$ has no simple subquotients with highest weights outside of the set $\widetilde{W}_\Pg \cdot \lambda \cap \Pg_l$.

\paragraph{The De Concini-Kac quantum group.} 
 Let $\Ul_v^{\op{sc}}(\g)$ denote the extended
version of the quantum group $\Ul_v(\g)$, obtained by adjoining the generators $K_\mu , \mu \in \check{\Pg}$, satisfying the following relations: 
\[  K_\mu K_\nu = K_{\mu + \nu} ,  
\quad  K_\mu E_i  K_\mu^{-1} = v^{\langle \mu, \check{\alpha}_i \rangle} E_i , 
\quad K_\mu F_i  K_\mu^{-1} = v^{-\langle \mu, \check{\alpha}_i \rangle} F_i .\] 

Then we have $K_i := K_{d_i \check{\alpha}_i}$.
The {\it De Concini-Kac quantum group} $\Ul \equiv \Ul^{\op{DCK}}$ is obtained from $\Ul^{\op{sc}}(\g) $ by the base change: 
\[ \Ul \equiv \Ul^{\op{DCK}} = \Ul^{\op{sc}} (\g) \otimes_{\Q[v, v^{-1}]} \Q(q) , \]
sending $v$ to a primitive $l$th root of unity $q \in \C$. 

Let $\Zl$ denote the center of $\Ul$. Then $\Zl$ contains the Harish-Chandra center $\Zl_{HC}$, obtained from the center of 
the generic quantized enveloping algebra by specialization, and the $l$-center $\Zl_{l}$, generated by the $l$-th powers 
of the generators $E_i^l, F_i^l, K_\mu^{\pm l}$ where $i \in I$ and $\mu \in \check{\Pg}$. In fact, by the work of \cite{dCKac}, one has $\Zl = \Zl_{l} \otimes_{\Zl_{l} \cap \Zl_{HC} } \Zl_{HC}$.  


\paragraph{Singular block in the representation category over $\ul$.} 
For $\lambda \in \bar{\Cg}$ denote by $\chi_\lambda : Z_{HC} \to \Q(q)$ the central character given by the specialization at the root of unity of the projection of $Z_{HC}$ to the Cartan part of $\Ul$, evaluated at 
 (any of the) highest weights in the set $(W \ltimes l\Qg) \cdot \lambda \cap \Pg^+$.  Let $\mathrm{Rep}(\Ul)^{{\lambda}}$ (resp.~$\mathrm{Rep}(\Ul)^{\hat{\lambda}}$) denote the full subcategory of the category of the finite dimensional $\Ul$-modules killed by the ideal determined by the central character $\chi_\lambda$ 
(respectively, by some power of this ideal). For example, if $\lambda =0 \in \bar{\Cg}$,  $\mathrm{Rep}(\Ul)^{0}$ contains the modules killed by the augmentation ideal in $Z_{HC}$. 

 Similarly, let $\mathrm{Rep}(\Ul)_0$ (resp. $\mathrm{Rep}(\Ul)_{\hat{0}}$) denote the full subcategory of $\Ul$-modules killed 
 by the augmentation ideal in the $l$-center $\Zl_l$ (respectively, by some power of this ideal). 
 Then the category $\mathrm{Rep}(\Ul)_0$  is equivalent to the representation category of modules over the finite-dimensional algebra obtained by factoring $\Ul$ over the augmentation ideal in $\Zl_l$, which differs from $\ul$ only in the Cartan part. It follows that we can identify the block over $\ul$ corresponding to the weight $\lambda$, which we denote by $\mathrm{Rep}(\ul)^\lambda$, as 
\begin{equation}
\mathrm{Rep}(\ul)^{\lambda}=\mathrm{Rep}(\Ul)_0^{\hat{\lambda}}.
\end{equation}

In particular, we have for the principal block $\mathrm{Rep}(\ul)^{0}=\mathrm{Rep}(\Ul)_0^{\hat{0}}$. In this paper, however, we will concentrate our attention on the case $\lambda \in \bar{\Cg}_{\rm sing}$.

\paragraph{Geometric description of the singular blocks.} 
In this subsection, we closely follow the arguments of  \cite[Section 2.4]{BeLa}.

The parabolic \emph{Grothendieck-Springer resolution} in  $\widetilde{\g}_P=G\times_P\p$, which fits into the short exact sequence
\begin{equation}
0 \lra G\times_P \n_P \lra G\times_P \p \lra G\times_P \mathfrak{l}_P\lra 0,
\end{equation}
so that $\Nt_P\subset \widetilde{\g}_P$ as a closed subvariety. Furthermore, $\widetilde{\g}_P$ can also be described as
\begin{equation}
\widetilde{\g}_P=\{
(Q,x)| ~Q\in G/P,~x\in \mathrm{Lie}(Q) \subset \g
\},
\end{equation}
where we have interpreted $Q$ as parabolic subgroups in $G$ that are isomorphic to $P$. 
As such, it is equipped with a natural $G$-equivariant projection map
\begin{equation}
\nu:\widetilde{\g}_P\lra \g ,\quad (Q,x)\mapsto x.
\end{equation}
so that $\widetilde{\g}_P$ sits inside the trivial bundle $G\times_P \g\cong \g\times X$ as a subbundle. We have the commutative diagram:
\begin{equation}
\begin{gathered}
\xymatrix{
\widetilde{\g}_P\ar[r]^-{\iota}\ar[d] & \g\times X \ar[r]^-{\pi_1}\ar[d] & \g\ar[d]\\
X\ar@{=}[r]\ar@/_1pc/[u]_{0} & X\ar[r]\ar@/_1pc/[u]_{0} & \mathrm{pt}\ar@/_1pc/[u]_{0}
}
\end{gathered} \ .
\end{equation}
Here the $0$s stand for the zero sections of the (vertical) bundle projections, and $\nu=\pi_1\circ \iota$.

Following the main results of \cite{BaKr,BaKrSing}(see also \cite{BMR1, BMR2} for the modular case), one can obtain the following geometric description of the derived category of the singular block of 
\begin{equation}
D^b(\mathrm{Rep}(\Ul)^{\hat{\lambda}}_{\hat{0}})\cong D^b(\mathrm{Coh}_{X}(\widetilde{\g}_P))
\end{equation} 
where $P$ is the parabolic subgroup whose Weyl group $W_P \subset W$ stabilizes $\lambda$, and   $\mathrm{Coh}_{X}$ stands for the category of coherent sheaves set-theoretically supported on the zero section of the canonical projection map
\begin{equation}
\mathrm{pr}: \widetilde{\g}_P\lra X.
\end{equation}

Under the functor $\mathbf{R}\mathrm{Hom}_{D^b(\mathrm{Coh}_{\tilde{\g}_P})}(\mathcal{O}_X,-)$ the triangulated category $D^b(\mathrm{Coh}_{X}(\widetilde{\g}_P))$ is equivalent to the derived category of differential graded (DG) coherent sheaves on the DG scheme, known as the \emph{derived zero fiber}
$$\widetilde{\g}_P\times^{\mathbf{L}}_{\g\times X}(\{0\}\times X)\cong \widetilde{\g}_P\times^{\mathbf{L}}_{\g}\{0\},$$
whose structure sheaf of DG algebras equal to $\mathcal{O}_{\widetilde{\g}_P}\otimes_{\mathcal{O}_\g}^{\mathbf{L}}\mathcal{O}_{0}$. To find a representative for this structure sheaf of DG algebras, we can resolve the structure sheaf of the zero section $\{0\}\times X$ inside $\g\times X$ (e.g. use the usual Koszul resolution $\wedge^\bullet\g\otimes_\C S^\bullet(\g^*)\otimes_\C \mathcal{O}_X\cong \wedge^\bullet\g\otimes_\C \mathcal{O}_{\g\times X}$), and take the usual tensor product of $\iota_*(\mathcal{O}_{\widetilde{\g}_P})$ with the resolution over $\mathcal{O}_{\g\times X}$. Similarly as done in \cite[Section 2.4]{BeLa}, one computes that the sheaf of DG algebras is canonically quasi-isomorphic to the sheaf of DG algebras
\begin{equation}
G\times_P(\wedge^\bullet(\g/\p)^*)\cong \wedge^\bullet \Omega_{G/P}^1[1]
\end{equation}
with the zero differential, which is a version of the classical Hochschild-Rosenberg-Kostant Theorem.

%

It now follows by combining the discussions above that the category $D^b(\mathrm{Rep}(\ul)^{\lambda})$ can be identified with 
\begin{equation}
D^b(\mathrm{Rep}(\ul)^{\lambda})\cong D^{b}(\mathrm{Rep}(\Ul)^{\hat{\lambda}}_{0})\cong D^b_{\mathrm{DG}}(\mathrm{Coh}(\Omega_{X}[1])).
\end{equation}
%
%
Together with the usual Koszul duality between polynomial-exterior algebras, we obtain
\begin{equation}
D^b(\mathrm{Rep}(\ul)^\lambda)\cong D^b_{DG}(\mathrm{Coh}(\wedge^\bullet \Omega_{X}^1)[1])\cong D^b_{DG}(\mathrm{Coh}(S^\bullet TX)[-2])
\end{equation}

The relative $\mathrm{Spec}$ of the sheaf of commutative algebras $S^\bullet TX[-2]$ over $X$ is nothing but $\Nt_P$ with a $\C^*$-action rescaling the linear fibers according to the character $z\mapsto z^{-2}$. 

This discussion leads to the following theorem, which is the singular block analogue of the main result of \cite[Theorem 4]{BeLa}.

\begin{theorem}  \label{BeLa_cat} 
Let $\lambda \in \Cg_{\rm sing}$ be an integral weight of $G$, singular with respect to the dot action of the Weyl group 
$W$, and $P$ the parabolic subgroup whose Weyl group $W_P \subset W$ stabilizes $\lambda$. Then there is an equivalence of triangulated categories
\[
D^b(\mathrm{Rep}(\ul)^\lambda)\cong D^b(\mathrm{Coh}_{\C^*} (\Nt_P)).
\]
\end{theorem}

We record the consequence concerning the center of the singular blocks of $\ul$. This is the corresponding singular block analogue of \cite[Theorem 7]{BeLa} whose proof parallels that in \cite{BeLa}.

\begin{corollary} \label{BeLa_sing}
Let $\lambda \in \Cg_{\rm sing}$ be an integral weight of $G$, singular with respect to the dot action of $W$, and $P$ the parabolic subgroup whose Weyl group $W_P \subset W$ stabilizes $\lambda$. Then the Hochschild cohomology of the singular block $\mathrm{Rep}(\ul)^{\lambda}$ is isomorphic to the Hochschild cohomology ring of the variety $\Nt_P$:
\[
 \mathrm{HH}_{\C^*}^\star(\Nt_P)\cong\bigoplus_{i+j+k=\star}\mH^i(\Nt_P,\wedge^j T\Nt_P)^k.
\]
In particular, the center of the singular block $\zl^{\lambda}:=\zl(\ul^{\lambda})$ is isomorphic to the commutative algebra
\[
\zl^{\lambda}\cong \mathrm{HH}_{\C^*}^0(\Nt_P)\cong \bigoplus_{i+j+k=0}\mH^i(\Nt_P, \wedge^j T\Nt_P)^k.
\]
\end{corollary}

\subsection{The equivariant structure}\label{sec-equivariant-structure}
We start by considering the $G$-equivariant projection map 
\begin{equation}\label{eqn-pr}
\mathrm{pr}: \Nt_P\lra X.
\end{equation}
Upon differentiation, the $G$-actions induce a map of vector bundles that fit into the following diagram:
\begin{equation}\label{dgm-basic-diagram}
\begin{gathered}
\xymatrix{
\g\times \Nt_P\ar[r]\ar[d]_{\mathrm{Id}\times \mathrm{pr}} &  T\Nt_P\ar[d]^{d(\mathrm{pr})},\\
\g\times X\ar[r] & TX
}
\end{gathered} \ .
\end{equation}
Here $d(\mathrm{pr})$ stands for the total differential of the projection map. Notice that the lower horizontal part of Diagram \eqref{dgm-basic-diagram} is the middle-right part of the $G$-equivariant bundle sequence on $X$:
\begin{equation}\label{eqn-X-tan-bundle-sequence}
0\lra G\times_P \p\lra G\times_P\g \lra G\times_P\lu_P\lra 0.
\end{equation}
This is because $G\times_P \g$ is, in fact, a (non-equivariantly) trivial bundle pulled back from the trivial bundle $G\times_G \g\cong \{\mathrm{pt}\}\times \g$ over a point.

By adjunction, Diagram \eqref{dgm-basic-diagram} gives us a map of $G$-equivariant bundles on $\Nt_P$
\begin{equation}\label{dgm-com-sq-N}
\begin{gathered}
\xymatrix{
\g\times \Nt_P\ar[r]\ar@{=}[d] &  T\Nt_P\ar[d],\\
\mathrm{pr}^*(\g\times X)\ar[r] & \mathrm{pr}^*(TX)
}
\end{gathered} \ .
\end{equation}\label{eqn-N-tan-bundle-sequence}
The right-hand vertical arrow in the diagram is also part of a natural sequence of equivariant bundles on $\Nt_P$: tangent vectors that have zero projection onto $\mathrm{pr}^*(TX)$ must be tangential to the fibers of the projection map $\mathrm{pr}$. The sequence is thus equal to
\begin{equation}
0\lra \mathrm{pr}^* (G\times_P\n_P) \lra T\Nt_P \lra \mathrm{pr}^* (G\times_P \lu_P) \lra 0,
\end{equation}
where we have identified $TX\cong G\times_P \lu_P$ and $T^*X\cong G\times_P\n_P$ as equivariant $G$-bundles on $X$.

Combining Diagram \eqref{dgm-com-sq-N} and equation \eqref{eqn-N-tan-bundle-sequence} results in a map of short exact sequences of $G$-equivariant vector bundles on $\Nt_P$:
\begin{equation}\label{eqn-map-ses}
\begin{gathered}
\xymatrix{
0 \ar[r] & \mathrm{pr}^*(G\times_P \p)\ar[r]^-{\iota}\ar[d]^{\mathrm{ad}} & \g \times \Nt_P \ar[r]\ar[d] & \mathrm{pr}^*(G\times_P\lu_P)\ar[r]\ar@{=}[d]& 0\\
0 \ar[r] & \mathrm{pr}^*(G\times_P \n_P) \ar[r] & T\Nt_P \ar[r] & \mathrm{pr}^*(G\times_P\lu_P)\ar[r] & 0
}\end{gathered}
\ .
\end{equation}
The equality of on the right-hand term implies that the left square is a push-out. In particular, the tangent bundle $T\Nt_P$ fits into a short exact sequence of $G$-bundles on $\Nt_P$:
\begin{equation}\label{eqn-bundle-sequence-on-Nt}
0 \lra \mathrm{pr}^*(G\times_P \p)\stackrel{\iota\oplus\mathrm{ad} }{\lra} \g\times \Nt_P\oplus \mathrm{pr}^*(G\times_P \n_P) \lra T\Nt_P \lra 0.
\end{equation}

We analyze the map $\mathrm{ad}$ more carefully. The fact that $P$ stablizes the identity coset of $X$ and $\mathrm{pr}$ is $G$-equivariant tell us that $P$ (and thus via differentiation, $\p$) acts on the fiber $\n_P$ of the projection map over the identity coset. The action is no other than the adjoint representation of $P$ (or $\p$) on the unipotent subalgebra $\n_P$, and hence the name $\mathrm{ad}$. By choosing a $G$-equivariant non-degenerate bilinear form on $\g$ (e.g., the Killing form), the map can by identified with the linear map
\begin{equation}
\mathrm{ad}: \p\lra \mathrm{End}(\n_P)\cong \lu_P\otimes \n_P,
\end{equation} 
which is unique up to a nonzero scalar. Here and below, we identify $\n_P$ with its own tangent space, and $\lu_P$ with the space of linear functions on $\n_P$.

Notice that the actions by $G$ on $\Nt_P$ and $X$ commute with the following $\C^*$-actions on $\Nt_P$ and $X$: on $\Nt_P$, the latter group acts on the base $X$ trivially, while it acts by rescaling each fiber of the projection map $\mathrm{pr}$ via the character $z\mapsto z^{-2}$. This allows us to enhance the above $G$-equivariant bundles on $\Nt$ and $X$ to $G\times \C^*$-equivariant ones. In particular, on $X$, such a bundle is just a graded $G$-bundle. For instance, the direct image sheaf.\footnote{Here and throughout the section we abuse the notations and identify bundles with their sheaves of local sections, whenever this does not cause any confusion. } 
$$\mathrm{pr}_*(\mathcal{O}_{\Nt_P})\cong G\times_P S^\bullet(\lu_P)$$
 is a sheaf of graded algebras, with the space $\lu_P$ identified with linear functions on the fiber $\n_P$ of $\mathrm{pr}$ over the identity coset. 

Pushing the sequence \eqref{eqn-bundle-sequence-on-Nt} down to $X$, we obtain the following equivariant bundle description of $\mathrm{pr}_*(T\Nt_P)$ as a graded module over $G\times_P S^\bullet(\lu_P)$.  We consider the composition map
\begin{equation}\label{eqn-Delta-map}
\Delta: S^\bullet(\lu_P)\otimes \p\stackrel{\alpha}{\lra} S^\bullet (\lu_P)\otimes \g \oplus S^{\bullet}(\lu_P)\otimes \lu_P\otimes\n_P \stackrel{\beta}{\lra} S^\bullet (\lu_P)\otimes \g \oplus S^{\bullet+1}(\lu_P)\otimes\n_P.
\end{equation}
Here $\alpha=\mathrm{Id}_{S^\bullet(\lu_P)}\otimes (\iota, \mathrm{ad})$ and $\beta$ is given by contracting linear functions in $\lu_P$ with the polynomials in $S^\bullet(\lu_P)$. In other words, $\beta$ is the symmetrization map from $S^\bullet(\lu_P)\otimes \lu_P$ onto $S^{\bullet+1}(\lu_P)$.

\begin{theorem}\label{thm-equ-structure-V1}
Let $G$ be a complex simple Lie group and $P$ be a parabolic subgroup. Denote by $X=G/P$ the associated partial flag variety and by $\Nt_P$ the cotangent bundle of $X$. Then the pushforward bundle $\mathrm{pr}_*(T\Nt_P)$ along the canonical projection map $\mathrm{pr}:\Nt_P\lra X$ is isomorphic to the infinite-rank graded $G$-equivariant bundle
\[
\mathrm{pr}_*(T\Nt_P) \cong G\times_P\frac{S^\bullet (\lu_P)\otimes \g \oplus S^{\bullet}(\lu_P)\otimes \n_P}{\Delta (S^\bullet(\lu_P)\otimes \p)},
\]
where elements in $S^k (\lu_P)$ ($k\in \N$) have degree $2k$, while $\mathrm{deg}(\g)=\mathrm{deg}(\p)=0$ and $\mathrm{deg}(\n_P)=-2$.
\end{theorem}
\begin{proof}
The proof is mostly contained in the discussions prior to the Theorem. We only need to verify the statement about gradings here. 

Since the $G$-actions on $\Nt_P$ and $X$ commutes with those of $\C^*$. The vector fields generated from $\g$ on $X$, and hence those from the subalgebra $\p$, are of $\C^*$-degree zero. The tangent vectors along $\n_P$ have degree $-2$, and the linear functions $\lu_P$ on $\n_P$ have degree $2$. This shows that the maps $\alpha$ and $\beta$ in \eqref{eqn-Delta-map} are homogeneous, and hence so is $\Delta$.
\end{proof}

We abbreviate the $P$-equivariant $S^\bullet(\lu_P)$-module in the Theorem by $V_1$:
\begin{equation}\label{eqn-V1}
V_1:=\frac{S^\bullet (\lu_P)\otimes \g \oplus S^{\bullet}(\lu_P)\otimes \n_P}{\Delta (S^\bullet(\lu_P)\otimes \p)}.
\end{equation}
It is clear that $V_1$ is free over $S^{\bullet}(\lu_P)$ of rank $2\mathrm{dim}_\C(X)$. Therefore we may form its exterior algebra over the polynomial ring $S^\bullet(\lu_P)$:
\begin{equation}\label{eqn-V-ext-alg}
V_\star:=\bigoplus_{k=0}^{2\mathrm{dim}_\C(X)} \wedge_{S^\bullet(\lu_P)}^{k} V_1.
\end{equation}
The $\C^*$-homogeneous degree-$r$ part of the $k$th wedge power will be denoted by $V_k^{r}$. Note that $r$ only takes value in even integers, with only finitely many of them negative.

\begin{corollary}\label{cor-wedge-P-structure}
The pushforward of the total exterior product of the tangent bundle $\wedge^\star T\Nt_P$ has the $G\times\C^*$-equivariant bundle structure as
\[
\mathrm{pr}_*(\wedge^\star T\Nt_P)\cong G\times_P V_\star,
\]
both sides regarded as sheaves of free $G\times_P S^\bullet(\lu_P)$-modules.
\end{corollary}
\begin{proof}
Consider the tensor product bundle $ T\Nt_P ^{\otimes k} $ equipped with the right $S_k$-action, which commutes with the $G\times \C^*$-action. Then the $k$th exterior product bundle is isomorphic to 
$$\wedge^kT\Nt_P\cong  T\Nt_P ^{\otimes k}\otimes_{\C[S_k]}\mathrm{sgn},$$ 
where $\mathrm{sgn}$ stands for the usual sign representation of the symmetric group. Since the projection map $\mathrm{pr}$ is affine, $\mathrm{pr}_*$ is an exact functor. Thus we have
\begin{eqnarray*}
\mathrm{pr}_* (T\Nt_P^{\otimes k}) & \cong & \mathrm{pr}_*(T\Nt_P)\otimes_{\mathrm{pr}_*(\mathcal{O}_{\Nt_P})}\mathrm{pr}_*(T\Nt_P)\otimes_{\mathrm{pr}_*(\mathcal{O}_{\Nt_P})} \dots \otimes_{\mathrm{pr}_*(\mathcal{O}_{\Nt_P})}\mathrm{pr}_*(T\Nt_P)\\
& \cong & G\times_P \left(V_1\otimes_{S^\bullet(\lu_P)}V_1\otimes_{S^\bullet(\lu_P)}\dots \otimes_{S^\bullet(\lu_P)} V_1\right),
\end{eqnarray*}
where there are $k$ tensor factors on the right hand side.
The result then follows by taking the signed invariants of $S_k$ and using that this functor is exact in characteristic zero and commutes with the $P\times \C^*$-action.
\end{proof}

\begin{example}\label{eg-some-G-bundle}
Using Corollary \ref{cor-wedge-P-structure}, we can write out the explicit equivariant bundle structure for any $\mathrm{pr}_*(\wedge^r T\Nt_P)^{s}$, with $r\in \N$ and $s\in 2\Z$. To do so, one can first consider the $r$th exterior product of $S^\bullet(\lu_P)\otimes\g\oplus S^\bullet(\lu_P)\otimes\n_P$ over the polynomial ring $S^\bullet(\lu_P)$, then mod out the ideal in the exterior algebra generated by $S^\bullet(\lu_P)\otimes \p$ under the map $\Delta$ (see equation \eqref{eqn-Delta-map}).

For instance, the bundle $\mathrm{pr}_* (\wedge^2 T\Nt_P)^{-2}$ admits the equivariant description as
\[
\mathrm{pr}_* (\wedge^2 T\Nt_P)^{-2} \cong
G\times_P \frac{\g\otimes\n_P\oplus \lu_P\otimes\n_P\wedge\n_P}{\Delta_2(\p\otimes \n_P)},
\]
where $\Delta_2$ is the composition map
\[
\Delta_2: \p\otimes \n_P \xrightarrow{\Delta \otimes \mathrm{Id}_{\n_P}}\g\otimes  \n_P \oplus \lu_P\otimes \n_P\otimes \n_P \stackrel{\alpha}{\lra} \g\otimes  \n_P \oplus \lu_P\otimes \wedge^2 \n_P
\]
where $\alpha$ is the anti-symmetrization map on the $\n_P\otimes \n_P$ factor.

For another example, consider the bundle $\mathrm{pr}_* (\wedge^3 T\Nt_P)^{-4}$. The Corollary gives us that it is isomorphic to
\[
\mathrm{pr}_* (\wedge^3 T\Nt_P)^{-4}\cong G\times_P
\left(
\frac{\g\otimes \n_P\wedge \n_P\oplus \lu_P\otimes \n_P\wedge \n_P\wedge \n_P}{
\Delta_4(\p\otimes \n_P\wedge \n_P)
}
\right).
\]
The map $\Delta_4$ is induced from $\Delta$ as a composition map
\[
\Delta_4: \p\otimes \wedge^2\n_P \xrightarrow{\Delta \otimes \mathrm{Id}_{\wedge^2 \n_P}}\g\otimes \wedge^2 \n_P \oplus \lu_P\otimes \n_P\otimes \wedge^2 \n_P \stackrel{\beta}{\lra} \g\otimes \wedge^2 \n_P \oplus \lu_P\otimes \wedge^3 \n_P
\]
where $\beta$, similar to $\alpha$, is the anti-symmetrization map on the $\n_P\otimes \wedge^2 \n_P$ factor.
\end{example}

Finally, we record the following well-known result to help us reducing the computation of sheaf cohomologies on $G/P$ to sheaf cohomologies on the full flag variety $G/B$.

\begin{lemma}\label{lem-reduce-to-flag}
Let $P$ be a parabolic subgroup containing a Borel $B$ inside a simple Lie group $G$, and let $V$ be a $P$-module. Then there is an isomorphism of $G$-representations
\[
\mH^i(G/P, G\times_P V)\cong \mH^i(G/B,G\times_B V).
\] 
\end{lemma} 
\begin{proof}
This follows from applying the projection formula and the Leray spectral sequence to the fiber bundle $\nu:G/B\lra G/P$:
\[
\oplus_{i+j=k}\mH^i (G/P, (G\times_PV)\otimes R^j\nu_*(\Ox_{G/B}))\Rightarrow \mH^k (G/B, \nu^* (G\times_P V) )= \mH^k(G/B, G\times_B V).
\] 
The fiber of $\nu$ over the identity coset $eP$ can be identified with the $P/B$. Since the fibers of $\nu$ are all smooth rational varieties, 
 \[
 R^j\nu_*(\Ox_{G/B})\cong 
 \left\{
\begin{array}{cc}
\Ox_{G/P} & j=0,\\
0 & j\neq 0.
\end{array} 
 \right.
 \]
The result follows. The claim about $G$-modules holds because of the functoriality of the projection formula and Leray spectral sequence.
\end{proof}

This lemma allows us to use the same approach as we did in \cite{LQ1} for the regular block, and in particular to apply the theorem of Bott that relates the sheaf cohomology of vector bundles over $X = G/B$ with the relative cohomology of Lie algebras $(\lb, \, \h)$ with coefficients in a certain $\lb$-module.

\begin{theorem}[Bott~\cite{B}] \label{Bott_rel_lie}  Let $E$ be a holomorphic $B$-module, and $W$ a holomorphic $G$-module. 
Let ${\mathcal E}$ be the sheaf of local holomorphic 
sections of the equivariant vector bundle $G\times_B E$ on the flag variety $X$. Then there is an isomorphism of vector spaces
\[ 
\op{Hom}_G \left( W, \; \mH^{\bullet}(X, \, {\mathcal E}) \right) = \mH^{\bullet}(\lb, \, \h, \,  \op{Hom}(W,E) ) =   
 \mH^\bullet(\n, \op{Hom}(W,E))^\h .\]
\end{theorem} 

Following \cite{LQ1}, we can further combine it with the combinatorics of the BGG resolution for a simple finite dimensional $\g$-module in order to derive a combinatorial method to compute the Lie algebra cohomology explicitly. The following statement is proven in \cite{LQ1}.

\begin{proposition}  \label{rel_lie_BGG}
Let $E$ be a finite-dimensional $\lb$-module, and $L_\nu$ be the finite-dimensional simple $\g$-module of dominant highest weight $\nu$. Then the $\N$-graded multiplicity space
\[ 
\op{Hom}_G \left( L_\nu, \; \mH^{\bullet}(X, G\times_B E) \right)  
\] 
is given by the cohomology of the complex 
\begin{equation}  \label{BGG_complex}
0\lra E[\nu]\to \cdots \to \bigoplus_{l(w) =j-1} E[w\cdot \nu] \stackrel{d_j^*}{\longrightarrow} \bigoplus_{l(w) =j} E[w\cdot \nu]\to \cdots \to E[w_0\cdot \nu]\lra 0, 
\end{equation}
where $w_0$ stands for the longest element of the Weyl group for $G$ and $j\in\{1,2,\dots, l(w_0)\}$. \hfill$\square$
\end{proposition} 

See \cite{LQ1} for the definition of the maps $d_j^*$. In particular, all these maps are computed explicitly for $\g = \mathfrak{sl}_4$ in  \cite{LQ1}, Section 4.3.

\subsection{Symmetries on a singular block of the center}\label{sec-symmetries}

To start, we record the following result to be used later.

\begin{lemma}\label{lemma-some-coh-group}
Let $G$ be a simple Lie group, and $X=G/P$ be the partial flag variety associated to a parabolic group $P$, and denote respectively by $\Omega$ and $T$ the cotangent and tangent bundle of $X$. Then
\begin{itemize}
\item[(i)] $\mH^0(X,T\otimes\Omega )\cong \C$,
\item[(ii)] $\mH^0(X, T\otimes \Omega^2)=0$.
\end{itemize}
\end{lemma}
\begin{proof}
These results are well-known for homogeneous manifolds that admit K\"{a}hler-Einstein metrics. We defer the proof until Section \ref{sec-diagonal-spaces} (see Corollary \ref{cor-T-tensor-Omega}, Lemma \ref{lemma-stable-bundles-are-simple} and Lemma \ref{lemma-stablility-sym-wedge}).
\end{proof}

Since $\Nt_P$ is naturally a holomorphic symplectic variety with an exact symplectic form $\omega\in \mH^0(\Nt_P,\wedge^2 T^*\Nt_P)$ (see, for instance, \cite[Chapter I]{CG} for more details), the tangent bundle and cotangent bundle of the variety are isomorphic via contraction with $\omega$:
\[
\iota_\omega: T\Nt_P \stackrel{\cong}{\lra} T^*\Nt_P.
\]
It follows that $\wedge^2 T\Nt_P$ has a canonical section that is dual to the symplectic form $\omega$, which is the \emph{Poisson bivector field} $\tau$.

\begin{lemma}\label{lemma-uniqueness-Poisson}
The space $\mH^0(\Nt_P,\wedge^2T\Nt_P)^{-2}$ is one-dimensional, and it is spanned by the Poisson bivector field $\tau$ which is dual to the canonical holomorphic symplectic form $\omega\in \mH^0(\Nt_P, \wedge^2 T^*\Nt_P)$.
\end{lemma}
\begin{proof}
We want to compute the cohomology group
$$\mH^0(\Nt_P,\wedge^2T\Nt_P)^{-2}\cong \mH^0(X, \mathrm{pr}_*(\wedge^2 T\Nt_P)^{-2}).$$
To do so, we use the description of $\mathrm{pr}_*(\wedge^2 T\Nt)^{-2})$ in Example \ref{eg-some-G-bundle}, which tells us that the bundle fits into a short exact sequence of vector bundles
\[
0\lra G\times_P(\lu_P\otimes \n_P\wedge\n_P)\lra \mathrm{pr}_*(\wedge^2 T\Nt_P)^{-2}) \lra G\times_P (\g/\p\otimes\n_P)\lra 0,
\]
Identifying $\g/\p$ with $\lu_P$, and taking global sections of the sequence, we obtain, as part of a long exact sequence
\[
0\lra \mH^0(X, T\otimes \wedge^2 \Omega)\lra \mH^0(\Nt_P,\wedge^2 T\Nt_P)^{-2}\lra \mH^0(X, \Omega\otimes T)\lra \cdots
\]
The previous Lemma \ref{lemma-some-coh-group} thus gives us an upper bound for the dimension of the cohomology group to be at most one. The Lemma follows since the Poisson bivector field $\tau$ is a nowhere zero section that has degree $-2$ (the symplectic form $\omega$ has degree two by our choice of the $\C^*$-aciton).
\end{proof}

Using the Poisson bivector field and the symplectic form, we next construct an $\mathfrak{sl}_2$ action on the total Hochschild homology of $\Nt_P$. The action is reminiscent of the usual $\mathfrak{sl}_2$ action on the Dolbeault cohomology ring of a compact K\"{a}hler manifold, but much simpler. This is because wedging with $\tau$, the analogue of wedging with the first Chern class of an ample line bundle here, has an adjoint that is represented by contracting with the global holomorphic symplectic form $\omega$. 

\begin{theorem}
\label{thm-sl2-action} 
The exterior product with the Poisson bivector field $\tau$ and contraction with the symplectic form $\omega$ defines an $\mathfrak{sl}_2$ action on the total Hochschild cohomology of $\Nt_P$. In particular, for any $j\in \{0,\dots n\}$, the induced map of wedging with the Poisson bivector field $(n-j)$ times 
\[
\tau^{n-j}\wedge(-):\mH^i(\Nt_P, \wedge^j T\Nt_P)^{k} \lra \mH^i(\Nt_P, \wedge^{2n-j} T\Nt_P)^{k+2j-2n}
\]
is an isomorphism of finite-dimensional vector spaces.
\end{theorem}
\begin{proof} Consider the holomorphic bundle maps on the total exterior-algebra bundle $\wedge^\star T\Nt_P$:
\[
\iota_w:\wedge^\star T\Nt_P \lra \wedge^{\star -2 }T\Nt_P, \quad \eta \mapsto \iota_\omega \eta,
\]  
\[
\tau \wedge: \wedge^\star T\Nt_P \lra \wedge^{\star+2} T\Nt_P, \quad \eta \mapsto \tau\wedge \eta,
\]
where $\eta$ stands for a homogeneous local section of $\wedge^\star T\Nt_P$. 

A fiberwise computation shows that the commutator of these two operations satisfy
\[
[\iota_\omega, \tau \wedge] = n-\mathrm{deg},
\]
where $\mathrm{deg}(\eta)=k$ if $\eta$ is a local section of $\wedge^kT\Nt_P$. After suitable normalization, this is no other than the usual $\mathfrak{sl}_2$-action on the algebra-bundle $\wedge^\star T\Nt_P$. The result follows since the $\mathfrak{sl}_2$ action then descends to the cohomology groups.
\end{proof}

By the analogy with the usual Dolbeault cohomology of a compact K\"{a}hler manifold, we make the following definition for the holomorphic symplectic variety $\Nt_P$.

\begin{definition}\label{def-hodge-diamond}
The \emph{formal Hodge diamond} for the variety $\Nt_P$ is the following table of degree-zero Hochschild cohomology groups:
\begin{gather}\label{table-hodge}
\begin{array}{|c|c|c|c|c|} 
\hline
 \mH^0(\wedge^0T\Nt_P)^0    &                         &                          &                         \\  \hline 
 \mH^1(\wedge^1T\Nt_P)^{-2}    &  \mH^0(\wedge^2T\Nt_P)^{-2}   &                          &                        \\  \hline 
 \vdots                  &  \vdots                 &  \ddots                  &                       \\  \hline
 \mH^n(\wedge^n T\Nt_P)^{-2n}    &  \mH^{n-1}(\wedge^{n+1}T\Nt_P)^{-2n}   &  \hspace{0.3in}\dots\hspace{0.3in}         & \mH^{0}(\wedge^{2n} T\Nt_P)^{-2n}   \\  \hline
\end{array}  \ .
\end{gather} 
The empty entries indicate that the corresponding terms vanish for degree reasons. We will denote by $\zl_P$ the direct sum 
\[
\zl_P:=\bigoplus_{i,j\in \N}\mH^i(\Nt_P,\wedge^{j}T\Nt_P)^{-i-j},
\]
which is a commutative algebra bigraded by $(i,j)\in \N^2$. Note that $\zl_P = \zl^\lambda$ in Theorem \ref{BeLa_sing} for $P = P_\lambda$.  
\end{definition}

Theorem \ref{thm-sl2-action} indicates that the formal Hodge diamond has a $\Z_2$-symmetry with respect to reflecting about the anti-diagonal. In particular, the left-most column in Table \ref{table-hodge} and the bottom row are isomorphic to each other, the two parts intersecting in the space $\mH^n(\wedge^n T\Nt_P)^{-2n}$.

The $\Z_2$-symmetry can also be realized purely in terms of the isomorphism of the bundle $T\Nt_P\cong T^* \Nt_P$. This is discussed in the prequel to the current work \cite{LQ1}, and we refer the reader there for more details.

Via the $\mathfrak{sl}_2$ action, one can find a large subalgebra of $\zl_P$, which, as we will see in Section \ref{sec-proj-space}, coincides with $\zl_P$ when $X\cong \mathbb{P}^n$. 

Observe that the first column in Table \ref{table-hodge} constitutes a copy of the (parabolic) coinvariant subalgebra 
\begin{equation}\label{eqn-coinvariant-subalg}
\mathrm{C}_P:=\bigoplus_{k=0}^n\mathrm{C}_P^{k,k} =\bigoplus_{k=0}^{n}\mH^k(\Nt_P,\wedge^k\Nt_P)^{-2k}.
\end{equation}
 inside $\zl_P$. This is because, for any $k\in \{0,\dots, n\}$,
\[
\mathrm{pr}_*(\wedge^kT\Nt_P)^{-2k}\cong G\times_P (\wedge^k \n_P) \cong \Omega_X^k,
\]
and thus, by the classical result of Borel on the (Dolbeault) cohomology of compact homogeneous K\"{a}hler manifolds, we have
\[
\mH^i(\Nt_P, \wedge^k\Nt_P)^{-2k}\cong \pr^*(\mH^i(X,\Omega_X^k))\cong 
\left\{
\begin{array}{cc}
\mH^{2k}(X,\C) & i=k, \\
0  & i\neq k.
\end{array}
\right.
\]
The $\Z_2$-symmetry of Table \ref{table-hodge} shows that the bottom row is isomorphic to $C_P$ up to an obvious bigrading transformation.

We define  $\tau\mathrm{C}_P$  to be the bigraded commutative algebra generated by the copy of $\mathrm{C}_P$ that appears in the first column of Table \ref{table-hodge} and $\tau$ with relations:
\begin{equation}\label{eqn-sl2-subalgebra}
\tau\mathrm{C}_P:= \frac{\mathrm{C}_P [\tau]}{\{x\tau^l|x\in \mathrm{C}_P^{r,r},~ r,l\in \N,~ l+r>n\}}.
\end{equation}
The following result is then a direct consequence of Theorem \ref{thm-sl2-action}.

\begin{corollary}\label{cor-sl2-subalgebra} 
\begin{enumerate}
\item[(i)]The algebra $\tau\mathrm{C}_P$ is a natural $\mathfrak{sl}_2$-invariant subalgebra inside $\zl_P$. 
\item[(ii)]It contains a copy of the parabolic coinvariant subalgebra $\mathrm{C}_P$ and a copy of $\mathrm{C}_P$ as its socle. The two copies of $\mathrm{C}_P$ intersect in the one-dimensional subspace $\mH^n(\Nt_P, \wedge^n T \Nt_P)^{-2n}$. \hfill$\square$
\end{enumerate}
\end{corollary}

\section{Examples and further questions}\label{sec-examples}
\subsection{Projective spaces}\label{sec-proj-space}
In this section, we compute the zeroth Hochschild cohomology of the cotangent bundle of the complex projective space $\PS^n$. Our goal is to show that the ``formal Hodge diamond'' for $\Nt_P:=T^* \PS^n$ looks like the following lower triangular matrix with entries all $1$-dimensional:

\begin{gather}  \label{table-hodge-proj}
\begin{array}{|c||c|c|c|c|} \hline 
         \scriptstyle{ j+i=0 } & 1   &                         &                         &                          \\  \hline 
              \scriptstyle{ j+i=2 } & 1   &   1                 &                         &                          \\  \hline 
               \scriptstyle{ \vdots }   & \vdots   &  \vdots  &    \ddots  &                          \\  \hline
             \scriptstyle{ j+i=2n }    &  1   &  1  &  \ldots  &   1   \\  \hline \hline 
 \scriptstyle{h^{i,j}} &  \scriptstyle{ j-i=0 }       &     \scriptstyle{ j-i=2 }       &    \ldots   &  \scriptstyle{ j-i=2n }  
 \\ \hline 
\end{array}  
\end{gather}

Here we have used the notation
\begin{equation}\label{eqn-hodge-numbers}
 h^{i,j} :=  h^{i}(\wedge^{j}T\Nt_P):=\mathrm{dim}_\C (\mH^{i}(\Nt_P, \wedge^j T\Nt_P))^{-i-j}.
\end{equation}
In particular, the inclusion of the subalgebra $\tau\mathrm{C}_P\subset \zl_P$ (Corollary \ref{cor-sl2-subalgebra}) is an equality for $\Nt_P$.

The result will be shown in several steps. First we observe that, by Corollary \ref{cor-sl2-subalgebra}, the first column and the bottom row constitute two copies of the cohomology space of $\PS^n$, so that, for any $0\leq k\leq n$,
\begin{equation}\label{eqn-first-column}
\mH^k(\Nt_P, \wedge^kT\Nt_P)^{-2k}\cong \mH^{k}(\Nt_P,\wedge^{2n-k}T\Nt_P)^{-2n}\cong \C.
\end{equation}

Next we recall the classical Euler's sequence for the cotangent bundle (sheaf) on $\PS^n$:
\begin{equation}\label{eqn-Euler-sequence}
0\lra \Omega_{\PS^n}\lra \mathcal{O}_{\PS^n}(-1)^{n+1}\lra \mathcal{O}_{\PS^n}\lra 0.
\end{equation}
See, for instance, \cite[Section 5.7]{CG} for more details. When no confusion can be caused, we will drop the subscripts of the sheaves.

\begin{lemma}\label{lemma-vanishing-twisted-cotang-bundle}
Let $n\in \N$ be a natural number. Then, for any $0\leq r \leq n-1$, the following cohomology groups on $\PS^n$ vanish:
\[
\mH^{\star} (\PS^n, \Omega_{\PS^n}^{\otimes r} (-1))=0.
\]
\end{lemma}
\begin{proof}
Then case when $r=0$ is well-known. To obtain the result, we will generalize and use induction on the number $a$ below to show that
\[
\mH^{\star} (\PS^n, \Omega_{\PS^n}^{\otimes (r-a)} (-s-1))=0
\]
for all $a$ and $s$ satisfying $0\leq s \leq a \leq r$.

When $a=r$, the result is clear since $-n\leq -r-1 \leq -s-1\leq -1$, and the cohomology groups
\[
\mH^{\star}(\PS^n,\mathcal{O}_{\PS^n}(-m)) 
\]
vanish for all $m\in \{1,\dots, n\}$.

Assume the induction hypothesis is true for $a+1$. Tensoring the Euler sequence \eqref{eqn-Euler-sequence} with $\Omega^{\otimes (r-a-1)}(-s-1)$, we get a short exact sequence of vector bundles
\[
0\lra \Omega^{\otimes (r-a)}(-s-1)\lra \left(\Omega^{\otimes (r-a-1)}(-s-2)\right)^{\oplus(n+1)}\lra \Omega^{\otimes (r-a-1)}(-s-1)\lra 0.
\]
Taking the associated cohomology long exact sequence, we see that, by induction hypothesis, the cohomology groups for the middle and right-hand terms vanish. Therefore so does the cohomology for the left-hand term, as desired.
\end{proof}

\begin{lemma}\label{lemma-tensor-Omega-cohomology}
On the projective space $\PS^n$, for any $k\in \{0,\dots, n\}$, the cohomology group
\[
\mH^{r}(\PS^n, \Omega_{\PS^n}^{\otimes k}) = 
\left\{
\begin{array}{lc}
\C & r = k,\\
0 & r \neq k.
\end{array}
\right.
\]
\end{lemma}
\begin{proof}
We prove this by induction on $k$. When $k=0,1$, the result is clear for all $n\geq 0$.

Assume that $n\geq 1$ and the induction hypothesis is true for $k$. Tensoring the Euler sequence \eqref{eqn-Euler-sequence} with $\Omega^{\otimes(k-1)}$, we obtain
\[
0\lra \Omega^{\otimes (k+1)} \lra \left(\Omega^{\otimes k}(-1)\right)^{\oplus (n+1)}\lra \Omega^{\otimes k}\lra 0.
\]
By the previous Lemma \ref{lemma-vanishing-twisted-cotang-bundle}, the middle term has vanishing cohomology since $k+1\leq n$. Hence the connecting maps provide isomorphisms
\[
\mH^r(\PS^n,\Omega^{\otimes k})\stackrel{\cong}{\lra} \mH^{r+1}(\PS^n,\Omega^{\otimes (k+1)}).
\]
The result follows.
\end{proof}

Now we will establish the dimension Table \ref{table-hodge-proj}. The $r$th column of the table consists of cohomology groups
\[
\mH^{0}(\Nt_P,\wedge^{2r}T\Nt_P)^{-2r},\ \dots, \mH^{k}(\Nt_P, \wedge^{2r+k}T\Nt_P)^{-2r-2k}, \ \dots,
\mH^{n-r}(\Nt_P, \wedge^{n+r}T\Nt_P)^{-2n},
\]
where $k\in \{0,\dots, n-r\}$. 

Fix such a $k$. By the equivariant structure of $\pr_*(\wedge^{2r+k}T\Nt_P)^{-2r-2k}$ (Corollary \ref{cor-wedge-P-structure}), we know that it has a natural filtration whose subquotients are of the form
\begin{equation}\label{eqn-subquotients}
\mathcal{F}_l:=G\times_P(S^{r-l}\lu_P \otimes \wedge^{l}(\g/\p)\otimes \wedge^{2r+k-l}\n_P)\cong 
S^{r-l}T\otimes \wedge^lT \otimes \Omega^{2r+k-l},
\end{equation}
where $l$ ranges between $\mathrm{max}(0,2r+k-n)$ and $r$. Let $\mathcal{O}(K):=\wedge^n\Omega$ be the canonical bundle of $\PS^n$.
Using the isomorphism of sheaves
\[
\Omega^{2r+k-l} \cong \mathcal{H}om(\Omega^{n+l-2r-k}, \mathcal{O}(K))
\cong \wedge^{n+l-2r-k} T(K),
\]
the subquotient \eqref{eqn-subquotients} is then isomorphic to
\[
\mathcal{F}_l\cong 
S^{r-l}T\otimes \wedge^lT \otimes \wedge^{n+l-2r-k} T(K).
\]
Taking the $k$th cohomology of the sheaf, we obtain, via Serre duality, that the cohomology of the subquotients are isomorphic to
\[
\mH^k(\PS^n,\mathcal{F}_l)\cong
\mH^{n-k}(\PS^n,
S^{r-l}\Omega\otimes \Omega^l\otimes \Omega^{n+l-2r-k}
)^*.
\]
Notice that the bundle $S^{r-l}\Omega\otimes \Omega^l\otimes \Omega^{n+l-2r-k}\subset \Omega^{\otimes (n-k+l-r)}$ is naturally a direct summand. Therefore, by Lemma \ref{lemma-tensor-Omega-cohomology}, the group
\[
\mH^k(\PS^n,\mathcal{F}_l)\cong \mH^{n-k}(\PS^n,
S^{r-l}\Omega\otimes \Omega^l\otimes \Omega^{n+l-2r-k}
)^*=0
\]
if $l\neq r$. When $l=r$, we have
\[
\mH^k(\PS^n,\mathcal{F}_r)\cong 
\mH^{n-k}(\PS^n,
\Omega^r\otimes \Omega^{n-r-k}).
\]
As the inclusions of bundles $\Omega^{n-k}\subset \Omega^r\otimes \Omega^{n-r-k} \subset \Omega^{\otimes(n-k)} $ all split, we get,
by applying Lemma \ref{lemma-tensor-Omega-cohomology} again, that
\[
\C\cong \mH^{n-k}(\PS^n,\Omega^{n-k})\subset \mH^{n-k}(\PS^n,\Omega^r\otimes \Omega^{n-r-k}) \subset \mH^{n-k}(\PS^n,\Omega^{\otimes(n-k)})\cong \C.
\]
Hence we have shown
\begin{equation}\label{eqn-subquotient-cohomology}
\mH^{k}(\PS^n, \mathcal{F}_l)\cong 
\left\{
\begin{array}{lc}
\C & l=r,\\
0  & l\neq r.
\end{array}
\right.
\end{equation}
Since $ \mH^k(\PS^n, \pr_*(\wedge^{2r+k}T\Nt_P)^{-2r-2k})$ is bounded by the $k$th cohomology of its subquotient sheaves in any filtration (an easy induction exercise), we conclude that
\begin{equation}
\mathrm{dim}(\mH^{k}(\Nt_P, \wedge^{2r+k}T\Nt_P)^{-2r-2k})=\mathrm{dim}( \mH^k(\PS^n, \pr_*(\wedge^{2r+k}T\Nt_P)^{-2r-2k}))\leq 1.
\end{equation}
To show that the equality actually holds, we resort to the $\mathfrak{sl}_2$ action on the formal Hodge diamond (Theorem \ref{thm-sl2-action}), which tells us that
\[
\tau^r\wedge(-): \mH^k(\Nt_P,\wedge^{k}T\Nt_P)^{-2k}\lra \mH^{k}(\Nt_P, \wedge^{2r+k}T\Nt_P)^{-2r-2k}
\]
is an injection from the $k$th entry from top on the first column, the first space being isomorphic to $\C$ thanks to \eqref{eqn-first-column}. Thus we have established the following.

\begin{theorem}\label{thm-dim-1}
Let $n$ be a natural number and $\Nt_P:=T^*\PS^n$. Then the graded Hochschild homology group, for any $r\in \{0,\dots, n\}$ and $k=0, \dots, n-r$, we have
\[
\mH^k(\Nt_P, \wedge^{2r+k}T\Nt_P)^{-2r-2k}\cong L_0
\]
as $\mathfrak{sl}_n$-representations. \hfill$\square$
\end{theorem}

\begin{corollary}\label{cor-algebra-structure}
The degree-zero Hochschild cohomology of $T^*\PS^n$ is isomorphic, as a bigraded commutative algebra, to 
\[
\frac{\C[x,\tau]}{\{x^a\tau^b|a,b\in \N,\ a+b>n\}},
\]
where $\mathrm{deg}(x)=(1,1)$ and $\mathrm{deg}(\tau)=(0,2)$.
\end{corollary}
\begin{proof} 
Identify the cohomology ring of $\PS^n$ with $\C[x]/(x^{n+1})$. By Theorem \ref{thm-dim-1}, wedging the powers of $x$ with $\tau$ gives us the full Hochschild cohomology ring. The result follows.
\end{proof}

\begin{corollary}  \label{cor_proj_sp}
Let $\lambda \in \Cg_{\rm sing}$ be an integral weight singular with respect to the dot action of the Weyl group, and such that 
for the corresponding parabolic subgroup $P_\lambda$, we have $G/P_\lambda \simeq \PS^n$. Then the center $\zl^\lambda$ of the block $\ul^\lambda$ is isomorphic, as a bigraded communtative algebra, to 
\[ \zl^\lambda \simeq  \frac{\C[x,\tau]}{\{x^a\tau^b|a,b\in \N,\ a+b>n\}},
\]
where $\mathrm{deg}(x)=(1,1)$ and $\mathrm{deg}(\tau)=(0,2)$. \hfill $\square$
\end{corollary}

\subsection{Diagonal subspaces}\label{sec-diagonal-spaces}
In this subsection, we establish one general feature that holds for the Hochschild cohomology of $\Nt_P:=T^*(G/P)$, with $G$ being an (almost) simple Lie group and $P$ being a parabolic. Namely, we will prove the following result.

\begin{theorem}\label{thm-diagonal-ones}
Let $G$ be a simple complex Lie group. The Hochschild homology groups on the main diagonal of the formal Hodge diamond \eqref{table-hodge} have entries all isomorphic to the trivial $G$-representation $L_0$: for any $r\in \{0,1,\dots, n=\mathrm{dim}(G/P)\}$,
\[
\mH^0(\Nt_P,\wedge^{2r}T\Nt_P)^{-2r}\cong L_0.
\]
Furthermore, an element spanning this one-dimensional space is given by $\tau^r$, the $r$th power of the Poisson bivector field.
\end{theorem}

The proof of this result will rely on some basic facts from K\"{a}hler-Einstein geometry, which we recall now. The reader is referred to the book by \cite{HuyModuli} for more details.

Let $X$ be an $n$-dimensional complex K\"{a}hler manifold with an ample line bundle $\mathcal{L}$ whose first Chern class will be denoted $c_1(\mathcal{L})=H$. Recall that, given any coherent sheaf $\mathcal{E}$ on $X$ (see, for instance, \cite[Section 1.2]{HuyModuli}) the \emph{slope} of $\mathcal{E}$ is the number
\begin{equation}\label{eqn-slope}
\mu(\mathcal{E}):= \frac{c_1(\mathcal{E})\cdot H^{n-1}}{\mathrm{rank}(\mathcal{E})}.
\end{equation}
The vector bundle $\mathcal{E}$ is called \emph{slope stable} (resp. \emph{slope semistable}) if, for any subsheaf $\mathcal{F}\subset \mathcal{E}$, the inequality
\begin{equation}\label{eqn-slope-inequality}
\mu(\mathcal{F}) < \mu(\mathcal{E})\quad (\textrm{resp.}~\mu(\mathcal{F}) \leq \mu(\mathcal{E}))
\end{equation}
holds. 

Notice that this definition of (semi)stability depends on a choice of an ample class $H$ on $X$. However, when $\mathcal{E}$ is a vector bundle on $X$ which is equipped with a \emph{K\"{a}hler-Einstein metric}, i.e., a Hermitian metric whose Ricci curvature tensor is proportional to the K\"{a}hler metric, then the bundle is semistable with respect to any ample class $H$. Another useful fact that we will use is that any semistable vector bundle is always a direct sum of its stable summands of identical slope. For the proof of results, see \cite[Section 1.6]{HuyModuli}.

A family of vector bundles that admit K\"{a}hler-Einstein metrics are given by the tangent bundle $TX$ or cotangent bundle $\Omega_X$ of a compact homogeneous K\"{a}hler manifold $X=G/P$ with any homogeneous Hermitian metric.
Furthermore, when $G$ is a simple complex algebraic group, both $TX$ and $\Omega_X$ are stable since they are indecomposable. By taking their tensor products and also tensor product of the corresponding metrics, one can show that $TX^{\otimes k}\otimes \Omega_X^{\otimes l}$ and their direct summands are semistable for any $k,l\in \N$. It follows that any (indecomposable) direct summand of these tensor product bundles must be semistable (stable).

\begin{lemma}\label{lemma-stable-bundles-are-simple}
If $\mathcal{E}$ is a slope stable vector bundle on a compact K\"{a}hler manifold $X$, then $\mathcal{E}$ is simple in the sense that
\[
\mathrm{End}_{\Ox_X}(\mathcal{E})\cong \C.
\]
\end{lemma}
\begin{proof}
This is well-known. See, for instance, \cite[Corollary 1.2.8]{HuyModuli}.
\end{proof}

We also record some basic Chern class computations for later use.

\begin{lemma}\label{lemma-tensor-chern-class}
Let $\mathcal{E}_1$ and $\mathcal{E}_2$ be complex vector bundles on manifold $X$. Then
\[
c_1(\mathcal{E}_1\otimes \mathcal{E}_2)=\mathrm{rank}(\mathcal{E}_1)c_1(\mathcal{E}_2)+\mathrm{rank}(\mathcal{E}_2)c_1(\mathcal{E}_1).
\]
Consequently, for any vector bundle $\mathcal{E}$ on $X$, we have
\[
c_1(\mathcal{E}^{\otimes k})= k (\mathrm{rank}(\mathcal{E}))^{k-1}c_1(\mathcal{E}).
\]
\end{lemma}
\begin{proof}By the splitting principle, one may assume these bundles are direct sums of line bundles, and the lemma is reduced to counting the Chern roots of these line bundles.
\end{proof}

\begin{lemma}\label{lemma-wedge-chern-class}
Let $\mathcal{E}$ be a complex vector bundle of rank $r$ on a manifold $X$, then
\[
c_1(\wedge^l \mathcal{E}) = {r-1 \choose l-1} c_1(\mathcal{E}).
\]
\end{lemma}
\begin{proof}
Exercise.
\end{proof}

\begin{corollary}\label{cor-chern-class-computation}Let $X$ be a manifold and $\mathcal{E}$ be a complex vector bundle on $X$ of rank $r$. Then
\[
c_1(\mathcal{E}^{\otimes k} \otimes \wedge^l\mathcal{E}^*)=
(k-l)r^{k-1}{r\choose l}c_1(\mathcal{E}).
\]
\end{corollary}
\begin{proof}
Using the previous Lemmas \ref{lemma-tensor-chern-class} and \ref{lemma-wedge-chern-class}, we compute
\begin{eqnarray*}
c_1(\mathcal{E}^{\otimes k} \otimes \wedge^l \mathcal{E}^*) & = & {r\choose l}c_1(\mathcal{E}^{\otimes k})+r^{k}c_1(\wedge^l\mathcal{E}^*)\\
& = & {r \choose l} k r^{k-1}c_1(\mathcal{E})-r^k{r-1\choose l-1}c_1(\mathcal{E})\\
& = & r^{k-1}{r\choose l}(k-l)c_1(\mathcal{E}),
\end{eqnarray*}
as desired.
\end{proof}

From now on, we fix $X$ to be $G/P$ associated with a simple complex Lie group, and take $\mathcal{E}=TX$ which is then slope stable. We will also fix $H=c_1(TX)=-K_X$ to be the anti-canonical class which is ample, although this is not important by the discussion above.

\begin{corollary}\label{cor-T-tensor-Omega}
For any pair of integers $(k,l)$ satisfying $0\leq k <l \leq n=\mathrm{dim}(X)$, the vector bundle $TX^{\otimes k}\otimes \Omega_X^{\otimes l}$ has no nonzero global sections. In other words
\[
\mH^0(X,TX^{\otimes k}\otimes \Omega_X^l)=0.
\]
\end{corollary}
\begin{proof}
The bundle in question is semistable since it has a $G$-equivariant K\"{a}hler-Einstein metric, and its slope
\[
\mu(TX^{\otimes k}\otimes \Omega_X^l)=\frac{k-l}{n}(-K_X)^n
\] 
is negative if $k<l$ by Corollary \ref{cor-chern-class-computation}. Hence the bundle can not have any global section, since the image of such a global section in the bundle would otherwise be a subsheaf of non-negative slope. 
\end{proof}

\begin{lemma}\label{lemma-stablility-sym-wedge}
Let $X=G/P$ be the partial flag variety associated with a complex simple Lie group. Then, for any $k\in \N$, the bundles $S^k TX$ and $\wedge^k TX$ are stable.
\end{lemma}
\begin{proof}
Identify the bundle $TX$ with $G\times_P \lu_P$, we are then reduced to showing that $S^k(\lu_P)$ and $\wedge^k(\lu_P)$ are indecomposable $P$-modules. For the symmetric power case, it is easy since the symmetric powers of a maximal root vector generates $S^k(\lu_P)$ by applying negative root vector from $\p$, so that $S^k(\lu_P)$ is in fact cyclic as a $U(\p)$-module. Any direct summand of $S^{k}(\lu_P)$ would either be the entire space if it contains the highest weight vector, or be zero if it does not. The result follows.

To show the claim for the exterior product, we use Schur-Weyl duality. Namely, since $S_k$ acts on $\lu_P^{\otimes k}$ by permuting tensor factors, the $\p$-direct summands are equipped with an involution induced by tensoring with the sign character of $S_k$, which interchanges $S^k(\lu_P)$ and $\wedge^k(\lu_P)$ when $k\leq \mathrm{dim}(\lu_P)$. Since $S^k(\lu_P)$ is indecomposable, then so must be $\wedge^k(\lu_P)$ if it is nonzero.
\end{proof}

Now we are ready to prove our main result in this subsection.

\begin{proof}[Proof of Theorem \ref{thm-diagonal-ones}] We push forward the $2r$th exterior product bundle $\wedge^{2r} T\Nt_P$ onto $X$, and pick out the degree-$(-2r)$ component $ \mathrm{pr}_*(\wedge^{2r}T\Nt_P)^{-2r} $.
We are reduced to compute the sheaf cohomology of the latter bundle on $X$.

By Corollary \ref{cor-wedge-P-structure}, $\mathrm{pr}_*(\wedge^{2r}T\Nt_P)^{-2r}$ has an increasing filtration whose associated graded parts are isomorphic (cf. Equation \ref{eqn-subquotients}) to the bundles
$$\mathcal{G}_l:=G\times_P(S^{r-l}\lu_P \otimes \wedge^{l}(\g/\p)\otimes \wedge^{2r-l}\n_P)\cong 
S^{r-l}T\otimes \wedge^lT \otimes \Omega^{2r-l},$$
where $l$ takes value in $\{\mathrm{max}(0,n-2r),\dots, r\}$. Since $\mathcal{G}_l$ is evidently a subbundle in
$T^{\otimes r}\otimes \Omega^{2r-l}$, we have, by Corollary \ref{cor-T-tensor-Omega}, that it has no nonzero global sections if $l<r$:
\[
\mH^0(X,\mathcal{G}_l)=\mH^0(X, S^{r-l}T\otimes \wedge^lT \otimes \Omega^{2r-l})\subset \mH^0(X, T^{\otimes r}\otimes \Omega^{2r-l})=0.
\]
When $l=r$, we have
\[
\mH^0(X,\mathcal{G}_r)=\mH^0(X,\wedge^r T \otimes \Omega^r),
\]
which has exactly a one-dimensional global section space spanned by $\mathrm{Id}_{\Omega^r}$ by Lemma \ref{lemma-stable-bundles-are-simple} and Lemma \ref{lemma-stablility-sym-wedge}. 

It follows that the sheaf $\mathrm{pr}_*(\wedge^{2r}T\Nt_P)^{-2r}$ has at most a one-dimensional global section space, and it is indeed one-dimensional because
\[
\tau^r\in \mH^0(\Nt_P,\wedge^{2r}T\Nt_P)^{-2r}=\mH^{0}(X,\mathrm{pr}_*(\wedge^{2r}T\Nt_P)^{-2r})
\]
is a nonvanishing section. 

Finally, $\tau^r$ is $G$-equivariant since $\tau$ is. This is because the Hamiltonian $G$ action on $\Nt_P$ preserves the symplectic form $\omega$, and thus it also preserves the dual bivector field $\tau$.
\end{proof}

\subsection{Singular blocks of \texorpdfstring{$A_3$}{A3}.}\label{sec-sing-blocks-A3}
In this subsection we will determine the bigraded dimensions of all singular block center of the small 
quantum group $\ul_q(\s_4)$. 

The decomposition of $\ul_q(\s_4)$ into blocks $\ul^\lambda \subset \ul_q(\s_4)$ is parametrized 
by the orbits of the extended affine Weyl group $\widetilde{W}_\Pg = W \ltimes l\Pg$ action on the set of the restricted dominant weights $\Pg_l$. 
(This is a consequence of the quantum linkage principle \cite{APW, APW2}). Then the block corresponding to the orbit containing the weight $\mu \in \Pg_l$ consists of left $\ul_q(\s_4)$-modules whose simple composition factors are of the form $L(w \cdot \mu)$, where $w \in \widetilde{W}_\Pg$ and $w \cdot \mu \in \Pg_l$.

Let $\mu \in \Pg_l$ and consider the orbit $\widetilde{W}_\Pg \cdot \mu \subset \Pg$. If it contains only regular weights with respect 
to the dot Weyl group action $W\cdot$, then the corresponding block is regular and its center is isomorphic to that of the principal block $\ul^0$, the latter having been computed in \cite{LQ1}. 
If the orbit $\widetilde{W}_\Pg \cdot \mu \subset \Pg$ contains a $W \cdot$ - singular weight $\lambda \in \bar{\Cg}_{\rm sing}$, 
then the block 
is singular and corresponds (not uniquely) to the parabolic subgroup $P_\lambda \subset G$ whose 
Weyl group can be identified with the stabilizer subgroup of $\lambda$ in $W$ with respect to the dot action. 
Therefore to each 
singular block $\ul^\lambda$, $\lambda \in \bar{\Cg}_{\rm sing}$, one can associate (possibly more than one) 
 isomorphism classe of parabolic subalgebras $\p \subset \g$, such that the Weyl group $W_P \subset W$ 
stabilizes $\lambda \in \Cg_{\rm sing}$. We label the parabolic subgroups of $SL_4(\C)$ as follows. 
\begin{equation}
\begin{array}{c} 
P_0= B= 
\left( 
\begin{matrix} 
 * & 0 & 0 & 0 \\ 
 * & * & 0 & 0 \\  
 * & * & * & 0  \\ 
 * & * & * & *  \\   
\end{matrix} 
\right),
\quad 
P_1= 
\left( 
\begin{matrix} 
 * & * & 0 & 0 \\ 
 * & * & 0 & 0 \\  
 * & * & * & 0  \\ 
 * & * & * & *  \\   
\end{matrix} 
\right),
\quad
P_2= 
\left( 
\begin{matrix} 
 * & 0 & 0 & 0 \\ 
 * & * & * & 0 \\  
 * & * & * & 0  \\ 
 * & * & * & *  \\   
\end{matrix} 
\right),
\\ \\
P_3= \left( 
\begin{matrix} 
 * & 0 & 0 & 0 \\ 
 * & * & 0 & 0 \\  
 * & * & * & * \\ 
 * & * & * & * \\   
\end{matrix} 
\right),
\quad    
P_4= 
\left( 
\begin{matrix} 
 * & * & * & 0 \\ 
 * & * & * & 0 \\  
 * & * & * & 0  \\ 
 * & * & * & *  \\   
\end{matrix} 
\right), 
\quad        
P_5= \left( 
\begin{matrix} 
 * & 0 & 0 & 0 \\ 
 * & * & * & * \\  
 * & * & * & * \\ 
 * & * & * & *  \\   
\end{matrix} 
\right),   
\\ \\
P_6= 
\left( 
\begin{matrix} 
 * & * & 0 & 0 \\ 
 * & * & 0 & 0 \\  
 * & * & * & *  \\ 
 * & * & * & *  \\   
\end{matrix} 
\right),  
\quad  
P_7= G=  
\left(  
\begin{matrix} 
* & * & * & * \\ 
* & * & * & * \\  
* & * & * & * \\ 
* & * & * & * \\   
\end{matrix} \right). 
\end{array}
\end{equation}
Here the notation for $P_i$ indicates the subgroup of $SL_4(\C)$ with nonzero entries only in the starred positions. Below we present the computations of the bigraded dimensions for each of the corresponding blocks of the center in the order of increasing complexity. As before, we will use the notations 
$
X:=G/P
$
and
$
\Nt_P:=T^*X.
$

\paragraph{The Steinberg block.} The Steinberg block is the unique simple block containing the Steinberg module $L((l-1)\rho)$, whose corresponding parabolic subgroup is $P_7=G$. The block is Morita equivalent to the category of finite-dimensional $\C$-vector spaces, and thus its center is $1$-dimensional.  This is reflected in the fact that, geometrically, both $X=G/G$ and $\Nt_P=T^*X$ consist of a single point.

\paragraph{The blocks corresponding to the projective space.} If $P = P_4 $ or $ P_5$, then $X\cong  \PS^3$. By the computations of Section \ref{sec-proj-space}, the formal Hodge diamond for $\Nt_P$ appears as follows.
 \begin{gather}  \label{table-hodge-pr3}
\begin{array}{|c||c|c|c|c|} \hline 
         \scriptstyle{ j+i=0 } & 1   &                         &                         &                          \\  \hline 
              \scriptstyle{ j+i=2 } & 1   &   1                 &                         &                          \\  \hline 
               \scriptstyle{ j+i =4 }   & 1   &  1  &    1    &                          \\  \hline
             \scriptstyle{ j+i=6 }    &  1   &  1  &  1  &   1   \\  \hline \hline 
 \scriptstyle{h^{i,j}} &  \scriptstyle{ j-i=0 }       &     \scriptstyle{ j-i=2 }       &   \scriptstyle{ j-i=4 }  &  \scriptstyle{ j-i=6 }  
 \\ \hline 
\end{array}  
\end{gather}
where  
\[
h^{i,j} = h^i(\wedge^j T\Nt_P):= \mathrm{dim}(\mH^i(\Nt_P, \wedge^j T\Nt_P)^{-i-j}).
\]
Thus the dimension of the block of the center is $10$.

\paragraph{The blocks corresponding to a Grassmannian.} Let $P=P_6$. In this case, $X\cong \mathrm{Gr}(2,4)$.
The formal Hodge diamond (Definition \ref{def-hodge-diamond}) for $\Nt_P$ has the following bigraded dimension table.
\begin{gather}  \label{table-hodge-gr24} 
\begin{array}{|c||c|c|c|c|c|} \hline 
         \scriptstyle{ j+i=0 } & 1   &                         &                         &            &              \\  \hline 
              \scriptstyle{ j+i=2 } & 1   &   1                 &                         &           &               \\  \hline 
               \scriptstyle{ j+i =4 }   & 2   &  2   &         1 &        &                  \\  \hline 
               \scriptstyle{j + i = 6} &  1  & 2   & 2    & 1  &      \\ \hline 
             \scriptstyle{ j+i=8 }    &  1   &  1  &  2  &   1  & 1   \\  \hline \hline 
 \scriptstyle{h^{i,j}} &  \scriptstyle{ j-i=0 }       &     \scriptstyle{ j-i=2 }       &   \scriptstyle{ j-i=4 } &  \scriptstyle{ j-i=6 }   &  \scriptstyle{ j-i=8 }   \\ \hline 
\end{array}  
\end{gather}
It follows that dimension of the block of the center is $20$. 

The verification of this table may be carried out easily via machine computation based on Corollary \ref{cor-wedge-P-structure} and Proposition \ref{rel_lie_BGG} as before. However, in this particular case, to illustrate our method, we provide an explicit computation which is presented in Appendix \ref{sub_grassmannian}.

\paragraph{The blocks corresponding to two-step partial flag varieties.} Now consider $P=P_1 $, $ P_2$ and $P_3$.  Evidently, in this case $X$ is identified with the partial flag varieties 
$$
G/P_1\cong\op{Fl}(2,3,4) \cong G/P_3\cong \op{Fl}(1,2,4),\quad
\quad G/P_2\cong \op{Fl}(1,3,4).
$$ 
The three parabolic subgroups correspond, in fact, to equivalent blocks of $\ul_q(\s_4)$. Indeed, a $\widetilde{W}_\Pg \cdot$-orbit 
containing a weight $\lambda_1 \in \bar{\Cg}_{\rm sing}$ such that $\lambda_1$ is stabilized by the simple reflection $s_1 \in W$, 
also contains weights $\lambda_2 \in \bar{\Cg}_{\rm sing}$ and $\lambda_3 \in \bar{\Cg}_{\rm sing}$ that are stabilized, respectively, 
by the simple reflections $s_2$ and $s_3 \in W$ and vice versa. Therefore, the block in question can be labeled equivalently as $\ul^{\lambda_1}$, 
$\ul^{\lambda_2}$, or $\ul^{\lambda_3}$, and to compute its structure we can choose any of the flag varieties 
$\op{Fl}(2,3,4)$, $\op{Fl}(1,3,4)$, or $\op{Fl}(1,2,4)$. This is confirmed by the computation: the entries of the formal Hodge diamond are the same independently of the parabolic type. 

To compute the entries we use the method given in Corollary \ref{cor-wedge-P-structure} and 
Proposition \ref{rel_lie_BGG} programmed using python. Note that Lemma \ref{lem-reduce-to-flag} allows us to use the maps $d^*_j$ computed in \cite{LQ1}, even if we have to compute cohomology over $G/P$. 
 
We tabulate the result of the computation below, performed by a machine computation.
  \begin{gather}  \label{table-hodge-part-flags}
\begin{array}{|c||c|c|c|c|c|c|} \hline 
         \scriptstyle{ j+i=0 } & 1 &                         &                         &           &  &                \\  \hline 
              \scriptstyle{ j+i=2 }   & 2   &   1                 &                         &          &  &                 \\  \hline 
               \scriptstyle{ j+i =4 } & 3   &   3  &  1    &                     &  &      \\  \hline
             \scriptstyle{ j+i=6 }    &  3   &  5  &  3  &   1   &  &  \\  \hline  
               \scriptstyle{ j+i=8 }  &  2   &  4  &  5  &   3   & 1  &  \\  \hline  
               \scriptstyle{ j+i=10 } &  1  &  2  &  3  &   3   & 2 &  1 \\  \hline    \hline         
 \scriptstyle{h^{i,j}} &  \scriptstyle{ j-i=0 }       &     \scriptstyle{ j-i=2 }       &   \scriptstyle{ j-i=4 }  &  \scriptstyle{ j-i=6 }  
 &  \scriptstyle{ j-i=8 }   &  \scriptstyle{ j-i=10 }   \\ \hline 
\end{array}  
\end{gather}
The dimension of the center of such a block is hence $50$. 

A discussion providing a categorical interpretation of these results is formulated in Section \ref{subsec-conj}.    

\paragraph{Regular blocks.}
A regular block corresponds to the Borel subgroup $P_0=B$ and its structure has been described in \cite[Theorem 4.11]{LQ1}. The dimension of the block is 125.
\begin{equation}\label{table-Hodge-sl4}
\begin{gathered}
\begin{array}{|c||c|c|c|c|c|c|c|} 
\hline
 {\scriptstyle i+j=0}        & 1          &            &            &            &            &            & \\ \hline 
 {\scriptstyle i+j=2}        & 3          &  1         &            &            &            &            & \\ \hline 
 {\scriptstyle i+j=4}        & 5          &  4         &  1         &            &            &            & \\ \hline
 {\scriptstyle i+j=6}        & 6          &  9         &  4         &  1         &            &            & \\ \hline 
 {\scriptstyle i+j=8}        & 5          &  11        & 9          &  4         &  1         &            & \\ \hline
 {\scriptstyle i+j=10}       & 3          &  8         & 11         &  9         &  4         &  1         & \\ \hline 
 {\scriptstyle i+j=12}       & 1          &  3         & 5          &  6         &  5         &  3         &  1 \\ \hline \hline
 {\scriptstyle h^{i,j}} & {\scriptstyle j-i=0} & {\scriptstyle j-i=2} & {\scriptstyle j-i=4} & {\scriptstyle j-i=6} & {\scriptstyle j-i=8} & {\scriptstyle j-i=10} & {\scriptstyle j-i=12} \\  \hline
\end{array}  
\end{gathered} 
\end{equation}

\subsection{Conjectures and comments } \label{subsec-conj}
In this section, we formulate two conjectures concerning the structure and dimension of the center of small quantum group $\ul_q(\mathfrak{sl}_n)$, and summarize the computational implications of the previous subsection. 

Let $n$ be a natural number greater or equal to $2$. For this section, we will assume that $q$ is a primitive $l$th root of unity, where $l$ is a prime number which is greater than $n$.

\paragraph{The trivial-representation conjecture.} We first point out the following observation. Since $\Nt_P$ is a $G$-variety and the bundle of polyvector fields $\wedge^\star T\Nt_P$ is $G$-equivariant, the $\C^*$-isotypic components of the Hochschild cohomology groups carry natural $G$-actions.

When $G=SL_n(\C)$ ($n\leq 4$), the computations performed in the previous section shows that the $SL_n(\C)$ action on the zeroth Hochschild cohomology group is always trivial. One is thus naturally led to expect the following.

\begin{conjecture}\label{conj-trivial-rep}
At a primitive $l$th root of unity $q$, the center $\zl(\mathfrak{sl}_n)$ of the small quantum group $\ul_q(\mathfrak{sl}_n)$ carries a natural $SL_n(\C)$ action, whose isotypic components consist only of trivial $SL_n(\C)$-representations.
\end{conjecture}

The conjecture generalizes the one formulated in \cite[Conjecture 4.7]{LQ1} for regular blocks of $\ul_q(\mathfrak{sl}_n)$. In contrast, the statement fails outside of type $A$, as the simplest instance in $B_2$ indicates (see Appendix \ref{app-b2}).

\paragraph{The dimension conjecture.}
Analyzing the results established in the previous subsections, one is naturally led to make the following conjecture.
\begin{conjecture} \label{conj-main}
Let $l$ be a prime number greater than $n+1$. 
At a primitive $l$th root of unity $q$, the dimension of the center of the small quantum group $\ul_q(\mathfrak{sl}_n)$ is given by the rational Catalan number 
\[
 c_{(n+1)l -n,\, n}  = \frac{1}{(n+1)l} { (n+1)l \choose n}~.
\] 
\end{conjecture}

The conjecture is verified for $\s_n$ for $n \leq 4$. To show this, we will need the following result, which holds for any $\s_n$. 

\begin{proposition}   \label{prop_number_blocks}
The number of regular blocks in $\ul_q(\mathfrak{sl}_n)$ at an $l$-th primitive root of unity is given by the 
rational Catalan number 
\[ c_{l-n, \,n}   = \frac{1}{l} { l \choose n}  , \] 
where $l$ is on odd integer greater than or equal to $n+1$ and coprime to $n$. 
\end{proposition} 

\begin{proof}  Consider the integral weights in the (open) fundamental $l$-alcove,  
\[ 
\Cg = \{ \mu \in \Pg^+ : \langle \mu + \rho, \check{\alpha} \rangle < l , \;\;\; \forall \alpha \in \Phi^+ \} 
\]  
This set is in bijection with the regular orbits of the affine Weyl group $\widehat{W}_\Qg = W \ltimes l\Qg$ on $\Pg$. 
Then the root weights $ \Cg \cap \Qg$ inside $\Cg$ are in bijection with the regular orbits of the extended affine Weyl group $\widetilde{W}_\Pg$.  
To find the cardinality of the set $|\Cg \cap \Qg|$, we recall that, in type $A_{n-1}$, 
the set $\Cg$ can be described also as follows: 
\[ 
\Cg = \{ \mu \in \Pg^+ : \langle \mu + \rho, \check{\theta} \rangle < l \} ,
\] 
where $\theta \in \Phi^+$ is the maximal root. For type $A_{n-1}$ we have 
$\langle \rho, \check{\theta} \rangle = n-1$, and therefore $\Cg = \{ \mu \in \Pg^+ : \langle \mu, \check{\theta} \rangle \leq l-n \}$ is the $(l-n)$-dilated fundamental $1$-alcove $\Cg_1 = \{ \mu \in \Pg^+ : \langle \mu, \check{\theta} \rangle \leq 1 \}$. 
Then by a result in \cite{GMV}, the  number of the root weights in $\Cg$ is equal to the number 
of the increasing $(l-n)/n$ parking functions, and this number is known to be given by the rational Catalan number 
$c_{l-n,\, n}$  \cite{GMV}. 
 \end{proof}

\paragraph{Case $\s_2$.} Let $l$ be an odd integer, greater than or equal to $3$. Then 
for $\s_2$ we have $(l-1)/2$ regular blocks with the dimension of the center equal to $3$, and one Steinberg block with $1$-dimensional center. The dimension of the entire center is 
\begin{equation}
{\op{ dim}}\, \zl (\s_2) =  3  \frac{l-1}{2}   + 1  = \frac{3l-1}{2} = c_{3l-2, \, 2} .
\end{equation}

\paragraph{Case $\s_3$.} Let  $l$ be a prime such that $l\geq 5$. Then by Proposition 
\ref{prop_number_blocks} for $\s_3$ we have $\frac{(l-1)(l-2)}{6}$ regular blocks of dimension $16$ each 
\cite{LQ1}. We also have $(l-1)$ singular $6$-dimensional blocks corresponding to the parabolic subalgebra $P \subset G$
such that $G/P \simeq \PS^2$ (see Corollary \ref{cor_proj_sp}), and one Steinberg block of dimension $1$. Therefore, 
\begin{equation}
 {\op{ dim}} \, \zl (\s_3) = 16 \frac{(l-1)(l-2)}{6} + 6 (l-1) + 1 = \frac{(4l-1)(4l-2)}{6} = c_{4l-3,\, 3} .
\end{equation}
This result agrees with the one obtained with the help of a machine computation done by an entirely different (straightforward) method in \cite{JackLa} when $l=5$. 

\paragraph{Case $\s_4$.} Let $l$ be a prime such that  $l \geq 7$. Then by Proposition \ref{prop_number_blocks} for $\s_4$ we have 
$\frac{(l-1)(l-2)(l-3)}{24}$ regular blocks of dimension $125$ each \cite{LQ1}. Let us determine the number of singular blocks 
of each kind. 

We will use the notations of the previous subsection. The blocks corresponding to any of the parabolic subalgebras 
$P_1\simeq P_3$ and $P_2$ are in bijection with the number of integral weights on the facets, 
but not on the edges of the closed first fundamental $l$-alcove $\bar{\Cg}$, which gives the number 
$\frac{(l-1)(l-2)}{2}$. The blocks 
corresponding to $P_4 \simeq P_5$ are in bijection with the integral weights on any of the two shorter edges, but not the 
vertices  of the tetrahedron $\bar{\Cg}$, a total of $l-1$ blocks. The blocks corresponding to $P_6$ are in bijection with 
half of the integral weights on the longer edge of the tetrahedron $\bar{\Cg}$, since there exists a transformation in  
$\widetilde{W}_\Pg$ that reflects this edge with respect to its middle. Therefore, the number of such blocks equals $\frac{l-1}{2}$. 
Finally, there is a one-dimensional Steinberg block. 

\begin{figure}[!htpb] 
\begin{center}
\caption{\small Three types of singular orbits for $\g= \s_4$, $l=7$}  \label{pic_orbits} 
\vspace*{1mm} 
\begin{tabular}{ccc} 
\psfig{figure=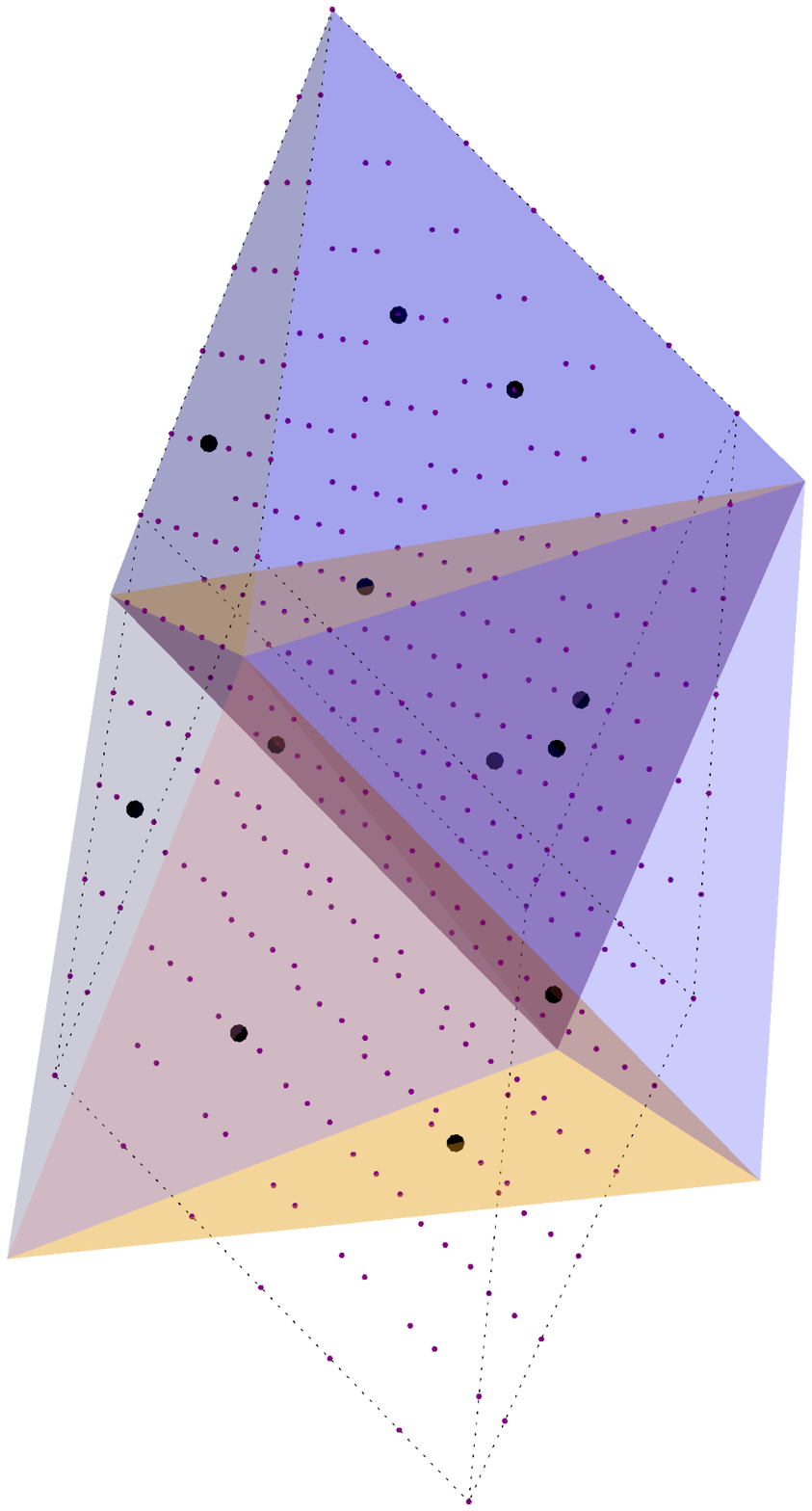, height=75mm}  
&  
\psfig{figure=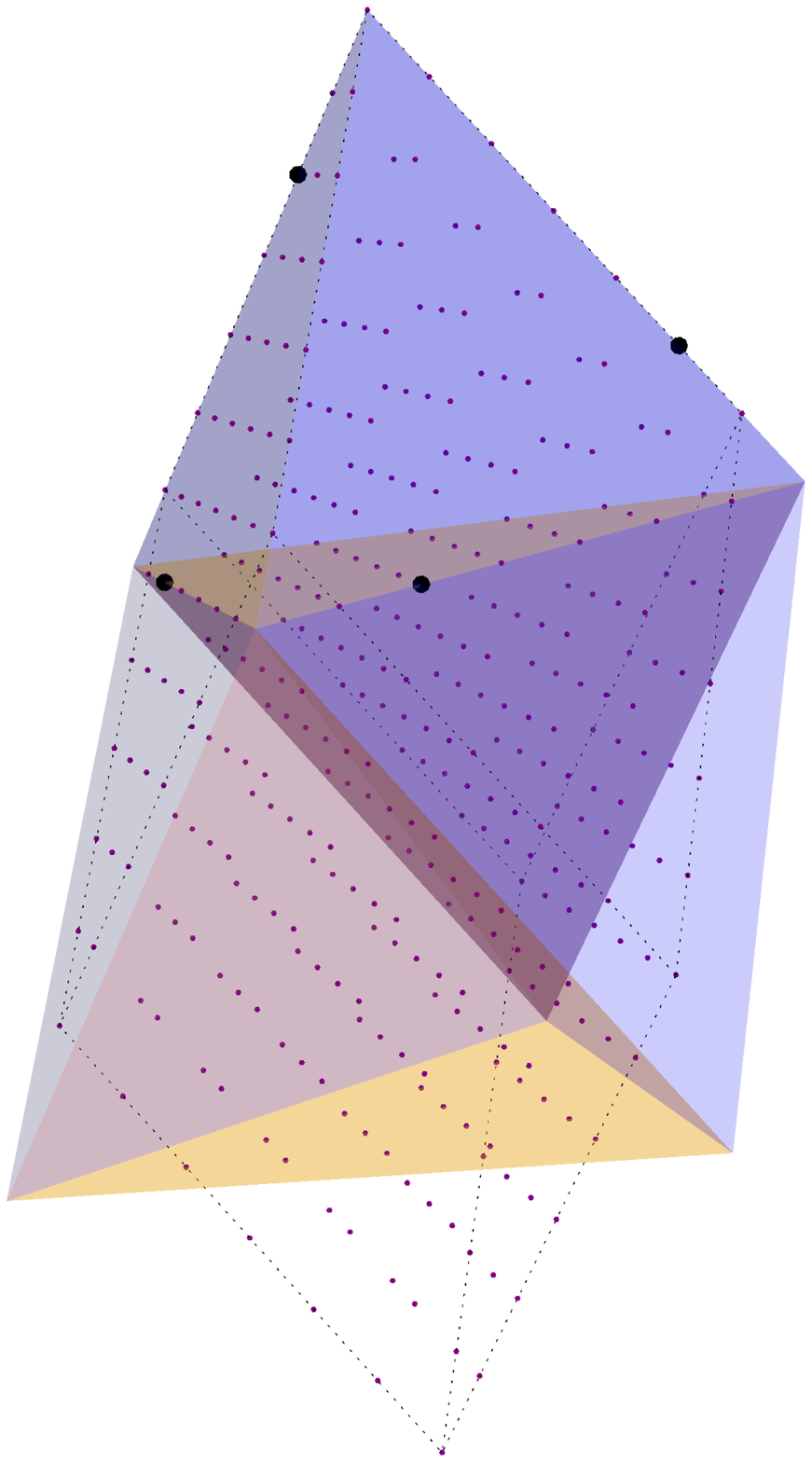, height = 75mm}
& 
\psfig{figure=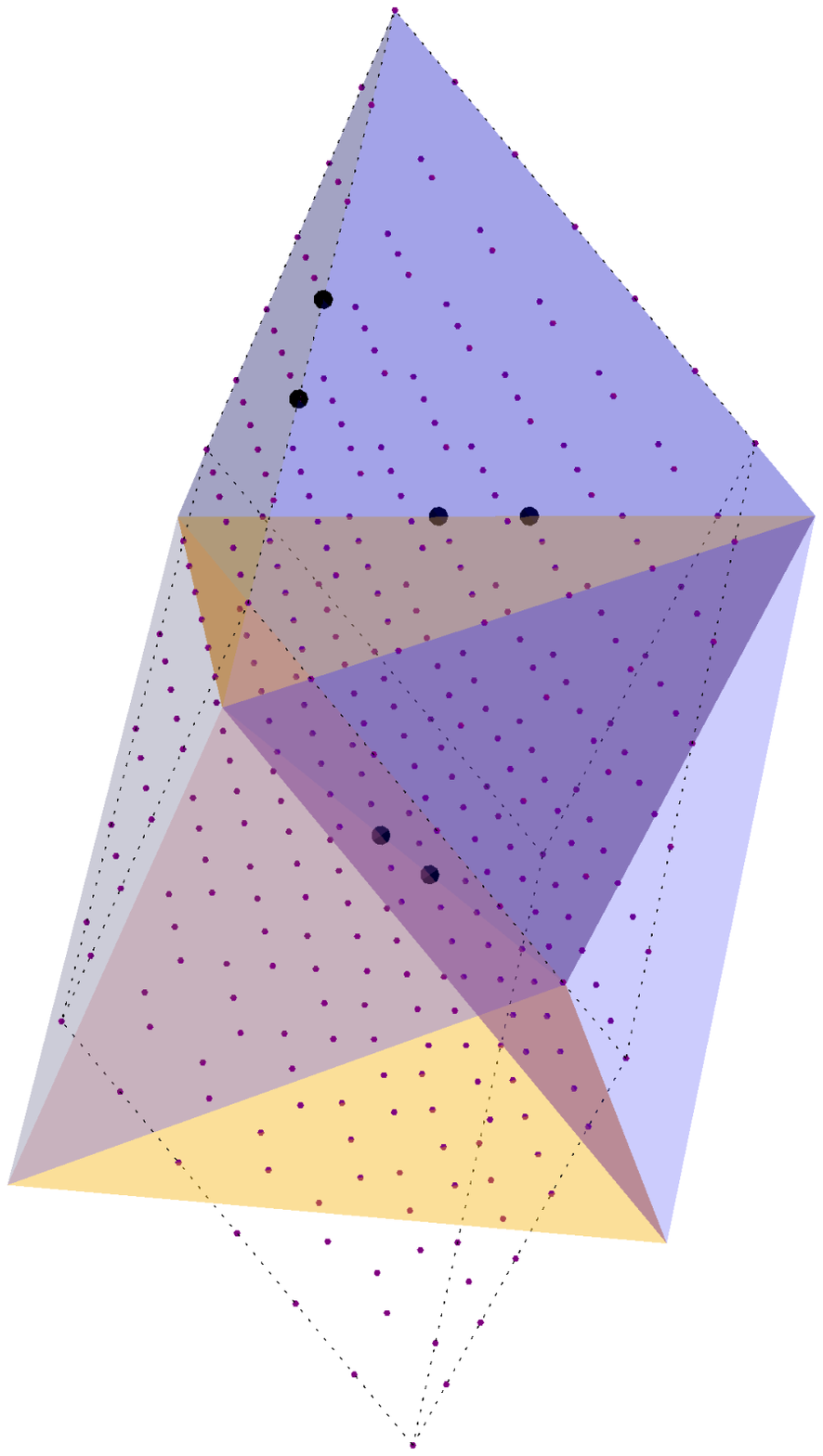, height = 75mm}  \\
{\small $P_1$-$P_2$-$P_3$ orbit} & {\small $P_4$-$P_5$ orbit} & {\small $P_6$ orbit}  
\end{tabular} 
\end{center}
\end{figure}

Figure \ref{pic_orbits} show the $\widetilde{W}_\Pg \cdot$-orbits of singular weights of different types in case $\g = \s_4$ and $l=7$. The small dots represent the set of all $l$-restricted dominant integral weights. We show only those faces of the alcoves 
that contain some of the $l$-restricted singular weights, for example, one can see that the lowest alcove contains only regular weights with respect to the $\widetilde{W}_\Pg \cdot$-action. 
The thick black dots represent the weights in one $\widetilde{W}_\Pg \cdot$-orbit: 12 weights on faces for the  orbit of the type $P_1-P_2-P_3$, 4 weights on short edges for the orbit of the type $P_4-P_5$, and 6 weights, two on each long edge for the orbit of the type $P_6$.

Taking into account the results on the dimensions of the singular blocks given in the previous section, we obtain
\begin{eqnarray}
{\op{dim}} \, \zl(\s_4) & \! = \! & 125 \frac{(l-1)(l-2)(l-3)}{24} + 50 \frac{(l-1)(l-2)}{2} + 20 \frac{l-1}{2}  + 10 (l-1) + 1  \nonumber \\
& \! =  \! & \frac{(5l-1)(5l-2)(5l-3)}{24} = c_{5l-4, \, 4} . 
\end{eqnarray}

\paragraph{A categorical implication.}
As we have observed in the case of two-step partial flag varieties, the center computations have resulted in a bigraded isomorphism of the formal Hodge diamonds 
 $$\mathrm{HH}^0_{\C^*}(T^*\mathrm{Fl}(2,3,4))\cong \mathrm{HH}_{\C^*}^0(T^*\mathrm{Fl}(1,3,4))\cong \mathrm{HH}_{\C^*}^0(T^*\mathrm{Fl}(1,2,4)).$$
While it is clear that, algebro-geometrically\footnote{The manifolds are diffeomorphic in the smooth category.}, the varieties $T^*\mathrm{Fl}(2,3,4)$, $T^*\mathrm{Fl}(1,3,4)$ are isomorphic, but neither is isomorphic to $T^*\mathrm{Fl}(1,2,4)$. These isomorphisms suggest that there might be a categorical equivalence of the corresponding varieties. 

Indeed, when a singular block of $\ul_q(\mathfrak{sl}_n)$ contains two weights $\lambda,\mu\in\overline{\Cg}_{\mathrm{sing}}$, then, even though the parabolic subgroups $W_\lambda$ and $W_\mu$ of $W$ fixing $\lambda$ and $\mu$ may be non-conjugate subgroups in $W$, they become conjugate to each other under the \emph{extended} affine Weyl group action $\widetilde{W}_\Pg$. Denote the parabolic subgroups determined by $W_\lambda$ and $W_\mu$ by $P_\lambda$ and $P_\mu$, and write $\Nt_\lambda:=T^*(G/P_\lambda)$ and $\Nt_\mu:=T^*(G/P_\mu)$ respectively. It follows from Theorem \ref{BeLa_cat} that we have the following equivalence of triangulated categories.

\begin{corollary}\label{cor-derived-coh-equivalence}
For $\ul_q(\mathfrak{sl}_n)$, let $\lambda$ and $\mu$ be singular weights in $\overline{\Cg}_{\mathrm{sing}}$ whose stabilizer groups under the $(-\rho)$-shifted finite Weyl group action are conjugate to each other in the extended affine Weyl group $\widetilde{W}_\Qg$. Then there is an equivalence of derived categories
\[
D^b(\mathrm{Coh}_{\C^*}(\Nt_\lambda))\cong D^b(\mathrm{Rep}(\ul)^\lambda)=D^b(\mathrm{Rep}(\ul)^\mu)\cong D^b(\mathrm{Coh}_{\C^*}(\Nt_\mu)).
\]
\hfill$\square$
\end{corollary}

It remains an interesting question to realize the equivalence of the derived category of coherent sheaves by Fourier-Mukai kernels.

We point out a connection relating the derived equivalence of Corollary \eqref{cor-derived-coh-equivalence} to the work of Bezrukavnikov \cite{Bezr}, Cautis-Kamnitzer-Licata \cite{CautisKamnitzerLicataQuiver} and Riche-Williamson \cite{RW}, which will be carried out in subsequent works.

Consider the setup formulated in \cite[Conjecture 59]{Bezr}. Let $\mathbf{O}:=\C[[t]]$ be the ring of formal power series, a subring in the field $\mathbf{K}:=\C((t))$. If $P$ is a parabolic subgroup in $G$, we set $I_P\subset G_\mathbf{O}$ be the parahoric subgroup which is the preimage of $P$ under the natural evaluation map $G_\mathbf{O}\lra G, t\mapsto 0$. Let $\mathrm{Fl}_{P}:= G_\mathbf{K}/I_P$ be the partial affine flag variety. If $P$ equals the Borel $B$, we simply write $I:=I_B$ and $\mathrm{Fl}:=\mathrm{Fl}_B$. The category of $I$-constructible mixed partial Wittaker sheaves on $\mathrm{Fl}_P$ will be denoted by $D_I(\mathrm{Fl}_P)$. Furthermore, if $G$ is a complex reductive group, let $G^\vee$ stand for the Langlands dual group of $G$ and similarly for $P^\vee$. Then it is expected that there is a (mixed) equivalence of triangulated categories
\begin{equation}\label{eqn-Bezr-equ}
D_{I}(\mathrm{Fl}_{P^\vee})\cong D^b(\mathrm{Coh}_{G\times \C^*}\Nt_{P}).
\end{equation}

In type $A$, (a singular version of) the work of \cite{RW} can be deployed to give a combinatorial description of the left hand side of equation \ref{eqn-Bezr-equ}. Such a description of the left hand-side by singular Soergel bimodules should only depend on the imput of the affine Weyl group $\widehat{W}_\Qg$ of $\g$ and a finite parabolic subgroup $W_P$ (the Weyl group of the Levi subgroup of $P$). Denote the bounded homotopy category of singular Soergel bimodules by $D_{W_P}$ (singular Soergel bimodules form only an additive category). One thus has a (conjectural) equivalence
\begin{equation}
D_{W_P}\cong  D^b(\mathrm{Coh}_{G\times \C^*}\Nt_{P}).
\end{equation}

In particular, in type $A$, if $\lambda$ and $\mu$ are singular weights of $\ul_q(\mathfrak{sl}_n)$ whose stabilizer subgroups $ W_\lambda $ and $ W_\mu $ are conjugate to each other under the extended affine Weyl group action $\widetilde{W}_\Pg$ on $\widehat{W}_\Qg$ as Dynkin diagram automorphisms, then the categories $D_{W_\lambda}$ and $D_{W_\mu}$ are equivalent to each other. This in turn induces an equivalence of triangulated categories of $G\times \C^*$-equivariant coherent sheaves, and gives rise to an equivalence of blocks of Lusztig big quantum group blocks. Forgetting the $G$-equivariance, one obtains the equivalence of Corollary \ref{cor-derived-coh-equivalence} concerning the small quantum group block. 

For the simplest example, the cases of two-step flag varieties considered in Section \ref{sec-sing-blocks-A3} all have stabilizer group $S_2$ fixing one simple root of the $A_3$; these subgroups are all conjugate to each other under the cyclic group $\widetilde{W}_{\Pg}/\widehat{W}_{\Qg}\cong \Z/4$ action. This explains why the center computations should have resulted in the same result a priori.

On the other hand, in type $A$, the cotangent bundle of partial flag varieties are special examples of Nakajima quiver varieties \cite{Na1}. In \cite{CautisKamnitzerLicataQuiver}, the authors constructed a categorified quantum $\Ul_v(\mathfrak{sl}_n)$ (a generic $v$) action on a certain fixed collection of quiver varieties determined by a given highest weight of $\mathfrak{sl}_n$. There one obtains a derived equivalence between various $D^b(\mathrm{Coh}_{\C^*}\Nt_{\lambda})$'s realized by a categorical quantum Weyl group action. It is natural to expect that the categorical quantum group action is related to the categorical Hecke algebra action in (the singular version of) \cite{RW} via categorified Schur-Weyl duality.

\appendix

\section{Some computations}
The appendix contains some computational details. The first subsection \ref{sub_grassmannian} gives some of the computations supporting the case of $T^*\mathrm{Gr}(2,4)$ in Section \ref{sec-sing-blocks-A3}, while the second subsection \ref{app-b2} exhibits the difference outside of type $A$ in the simplest case of type $B_2$.

\subsection{A Grassmannian example}  \label{sub_grassmannian}
In this subsection, we give some details for the computation that establishes Table \ref{table-hodge-gr24}. This is the first singular block example in this paper that needs the explicit $P$-modules structure of the pushforward polyvector field sheaves $\mathrm{pr}_*(\wedge^\star T\Nt_P)$.

To help us compute the cohomology groups, we will resort to a version of Borel-Weil-Bott Theorem tailored for  Grassmannians. We refer the reader to \cite{Jantzen} for details, and a very readable survey can also be found in \cite{Chip}.

For now, let $X:= \mathrm{Gr}(k,N)$ be the Grassmannian of $k$-planes in $\C^N$ ($0\leq k\leq N$), and let $S$ be the tautological bundle of $X$, which by definition assigns to a point $[P]\in X$ the plane $P$ it represents. We have a short exact sequence of vector bundles on $X$:
\begin{equation}\label{eqn-tautological-sequence}
0\lra S \lra \underline{\C}^N\lra Q\lra 0,
\end{equation}
where $\underline{\C}^N$ stands for the trivial vector bundle on $X$. Then the tangent and cotangent bundle of $X$ can be identified with
\begin{equation}\label{eqn-tan-cotan-bundle-Gr24}
TX\cong \mathcal{H}om_{\Ox_X}(S,Q)\cong S^*\otimes Q,\quad \quad \Omega_X\cong TX^*\cong Q^* \otimes S.
\end{equation}

Recall that if $V$ is a vector space of dimension $l$, and $\lambda=(\lambda_1,\lambda_2,\dots,\lambda_l)$ is an $l$-part partition of an integer $m\in \Z$, i.e., $\lambda_i\geq \lambda_{i+1}$ for all $i$, and $\sum_{i=1}^l\lambda_i=m$, the \emph{Schur functor} on $V$ assigns to $\lambda$ the $SL(V)$ representation $\Sch_\lambda(V)$ whose highest weight equals $\lambda$. When applied fiberwise to vector bundles, Schur functors produce new vector bundles on manifolds.

Now we apply Schur functors to the bundles $S$ and $Q$ on $X$. If $\alpha=(\alpha_1,\dots, \alpha_{N-k})$ is an $(N-k)$-part partition and $\beta=(\beta_1,\dots, \beta_{k})$ is a $k$-part partition, we obtain the bundle 
\[
\Sch_\alpha(Q)\otimes \Sch_\beta(S)
\]
on $X$. Let
\begin{equation}
\nu^\prime_{\alpha,\beta}=(\alpha_1+N,\dots ,\alpha_{N-k}+k+1, \beta_1+k,\dots, \beta_{k}+1)
\end{equation}
 be the shifted juxtaposition of the sequences $\alpha$ and $\beta$, i.e., we put $\alpha$ and $\beta$ side by side and add the sequence $(N,N-1,\dots, 1)$ to $(\alpha,\beta)$. Then choose $w_{\alpha,\beta}\in S_N$  such that the permuted sequence $s_{\alpha,\beta}\cdot \nu^\prime_{\alpha, \beta}$ becomes a partition, and we define
\begin{equation}\label{eqn-new-partition}
 \nu_{\alpha,\beta}:=s_{\alpha,\beta}\cdot \nu^\prime_{\alpha,\beta}-(N,N-1,\dots, 1).
\end{equation}
Also recall that the \emph{length} $|s|$ of a permutation $s\in S_N$  equals the number of inversions that occurs under the permutation of $s$ on $\{1,\dots, N\}$, i.e., $s(j)<s(i)$ but $1\leq i < j \leq N$ .

We will also use the transpose operation on partitions: $\Sch_{\alpha}(Q^*)=\Sch_{\alpha^*}(Q)$, where $\alpha^*$ is obtained from $\alpha$ as
\begin{equation}
\alpha^*=(-\alpha_{N-k}, \dots, -\alpha_1).
\end{equation}

\begin{theorem}[Borel-Weil-Bott]\label{thm-BWB-Gr}
Let $\alpha$ be an $(N-k)$-part partition and $\beta$ be a $k$-part partition. Then the vector bundle $\Sch_{\alpha}(Q)\otimes \Sch_{\beta}(S)$ on the Grassmannian $X:=\mathrm{Gr}(k,N)$ has at most one nonzero cohomology group:
\[
\mH^i(X,\Sch_{\alpha}(Q)\otimes \Sch_{\beta}(S))\cong 
\left\{
\begin{array}{cc}
\Sch_{\nu_{\alpha,\beta}}(V) & \textrm{if $\nu_{\alpha,\beta}$ is a partition and $i=|s_{\alpha,\beta}|$},\\
0 & \textrm{otherwise.}
\end{array}
\right.
\] 
Furthermore, the isomorphism is an identification of $SL_N$-representations.
\end{theorem}
\begin{proof}By pulling back to the flag variety, this reduces to the usual Borel-Weil-Bott Theorem for line bundles. See \cite[Section 5]{Chip} or \cite{Jantzen} for more details.
\end{proof}

We now establish the Table \ref{table-hodge-gr24} for $X=\mathrm{Gr}(2,4)$. First off, by Corollary \ref{cor-sl2-subalgebra}, the left most column and the bottom rows are identified with the cohomology of ring  $\mH^*(X,\C)$. Furthermore, by applying Theorem \ref{thm-diagonal-ones}, we know that the main diagonal of Table \ref{table-hodge-gr24} consists of one-dimensional spaces spanned by the powers of the Poisson bivector field $\tau$. Hence we are reduced to showing that
\begin{equation}
\label{eqn-Gr24-case1}
h^{2}(\wedge^4 T\Nt_P)=2,
\end{equation}
and 
\begin{equation}\label{eqn-Gr24-case2}
h^1(\wedge^3T\Nt_P)=h^1(\wedge^5 T\Nt_P)=2.
\end{equation}

\paragraph{Case I: $h^{2}(\wedge^4 T\Nt_P)$.} By the equivariant structure of 
\begin{equation}
\mathcal{F}_1:=\mathrm{pr}_*(\wedge^4 T\Nt_P)^{-6}\cong
G\times_P\left(\frac{\g\otimes \wedge^3 \n_P \oplus \lu_P\otimes\wedge^4\n_P}{\Delta (\p) \wedge (\wedge^3\n_P)}\right)
\end{equation}
on $X$ (Corollary \ref{cor-wedge-P-structure}), the bundle fits into a short exact sequence
\begin{equation}
0\lra T\otimes \Omega^4\lra \mathcal{F}_1\lra T\otimes \Omega^3\lra 0.
\end{equation}
Upon taking cohomology, we obtain
\begin{equation}\label{eqn-Gr24-les-1}
\cdots \lra \mH^2(T\otimes \Omega^4)\lra \mH^2(\mathcal{F}_1)\lra \mH^2(T\otimes \Omega^3)\lra \mH^3(T\otimes \Omega^4)\lra\cdots.
\end{equation}
To compute $\mH^2(T\otimes \Omega^4)$ and $\mH^2(T\otimes \Omega^3)$, we use the bundle isomorphisms $\Omega^4 = \mathcal{O}(K)$, which is the canonical bundle, and $\Omega^3\cong T(K)$, together with Serre duality to obtain
\[
\mH^i(T\otimes \Omega^4)\cong  \mH^{4-i}(\Omega)^* \quad \quad \mH^i(T\otimes \Omega^3)\cong \mH^i(T \otimes T(K) )\cong \mH^{4-i}(\Omega\otimes \Omega)^*.
\]
for any $i\in \{0,1,2,3,4\}$. Then the sequence \eqref{eqn-Gr24-les-1} reduces to
\begin{equation}
0=\mH^2(\Omega)\lra \mH^2(\mathcal{F}_1)\lra \mH^2(\Omega \otimes \Omega)\lra \mH^1(\Omega)\cong L_0.
\end{equation}
Next we compute $\mH^2(\Omega \otimes \Omega)$. Since $\Omega \otimes \Omega \cong (Q^*)^{\otimes 2}\otimes S^{\otimes 2}$ via equation \eqref{eqn-tan-cotan-bundle-Gr24}, we have
\[
\Omega^{\otimes 2}\cong \bigoplus_{\alpha,\beta\in \{(2,0), (1,1)\}}\Sch_{\alpha}(Q^*)\otimes \Sch_{\beta}(S),
\]
 each partition appearing with multiplicity one.

By the BWB Theorem \ref{thm-BWB-Gr}, it is easy to compute that
\[
\mH^2(\Sch_{(2,0)}(Q^*)\otimes \Sch_{(1,1)}(S))\cong L_0\cong  \mH^2(\Sch_{(1,1)}(Q^*)\otimes \Sch_{(2,0)}(S)),
\]
while
\[
\mH^2(\Sch_{(2,0)}(Q^*)\otimes \Sch_{(2,0)}(S))\cong 0 \cong  \mH^2(\Sch_{(1,1)}(Q^*)\otimes \Sch_{(1,1)}(S)).
\]
For instance, the bundle $\Sch_{(2,0)}(Q^*)\otimes \Sch_{(1,1)}(S)$ gives us the sequence
\[
(0,-2,1,1)\xrightarrow{+(4,3,2,1)}(4,1,3,2)\xrightarrow{s=(243)\in S_4}(4,3,2,1)\xrightarrow{-(4,3,2,1)}(0,0,0,0).
\]
The permutation $s=(243)$ indicates the usual permutation action of the second position going to the fourth, and the fourth going to the third, while the third moving back to the second. It has length $2$, and hence
\begin{equation}\label{eqn-eg-schur-functor}
\mH^2(\Sch_{(2,0)}(Q^*)\otimes \Sch_{(2,0)}(S))\cong L_{0},
\end{equation}
and 
\begin{equation}\label{eqn-coh-Omega-tensor-2}
\mH^2(\Omega\otimes \Omega)\cong L_0^{\oplus 2}\quad \quad \mH^i(\Omega \otimes \Omega)=0~(i\neq 2).
\end{equation}
as $SL_4(\C)$-representations.

It then follows that $\mH^2(\mathcal{F}_1)$ can consist at most of two copies of $L_0$. The fact that it indeed equals $L_0^{\oplus 2}$ holds because of the fact that
\[
\tau\wedge(-):  \mH^2(\Nt_P \wedge^2 T\Nt_P)(\cong L_0^{\oplus 2})\lra \mH^2(\Nt_P,\wedge^4 T\Nt_P)
\]
is injective (Corollary \ref{cor-sl2-subalgebra}). Equation \ref{eqn-Gr24-case1} now follows.

\paragraph{Case II: $h^{1}(\wedge^3 T\Nt_P)$.} Again, we utilize the explicit description of the $G$-bundle (see Corollary \ref{cor-wedge-P-structure} or Example \ref{eg-some-G-bundle})
\begin{equation}
\mathcal{F}_2:=\mathrm{pr}_*(\wedge^3 T\Nt_P)^{-4}
\cong G\times_P \left(\frac{\g\otimes \n_P\wedge\n_P\oplus \lu_P\otimes \n_P\wedge \n_P\wedge \n_P}{\Delta(\p)\wedge \n_P\wedge \n_P}\right),
\end{equation}
which tells us that $\mathcal{F}_2$ fits into a short exact sequence of vector bundles
\begin{equation}
0\lra T\otimes \Omega^3 \lra \mathcal{F}_2\lra T\otimes \Omega^2\lra 0.
\end{equation}
Since $X$ has dimension $4$, we have the isomorphism of bundles $T\otimes \Omega^3 \cong T\otimes T(K)$, and $T\otimes \Omega^2\cong T\otimes \wedge^2T (K)$. Via Serre duality, we obtain
\begin{equation}
\mH^i(T\otimes \Omega^3)\cong \mH^{4-i}(\Omega \otimes \Omega)^*,\quad \quad
\mH^{i}(T\otimes \Omega^2)\cong \mH^{4-i}(\Omega\otimes \Omega^2)^*.
\end{equation}
Thus the cohomology of $\mathcal{F}_2$ fits into a long exact sequence 
\begin{equation}\label{eqn-Gr24-les-2}
0\lra \mH^1(\mathcal{F}_2)\lra \mH^3(\Omega\otimes \Omega^2)^*\lra   L_0^{\oplus 2} \lra \mH^2(\mathcal{F}_2)\lra 0
\end{equation}
by using the cohomology of $\Omega\otimes \Omega$ computed in \eqref{eqn-coh-Omega-tensor-2}. 

We now turn to finding 
$\mH^\star(\Omega\otimes \Omega^2)$. It is an easy exercise to see that
\begin{equation}
\Omega\otimes \Omega^2 \cong 
\Sch_{(3,0)}(Q^*)\otimes \Sch_{(2,1)}(S)\oplus \left(\Sch_{(2,1)}(Q^*)\otimes \Sch_{(2,1)}(S)\right)^{\oplus 2} \oplus
\Sch_{(2,1)}(Q^*)\otimes \Sch_{(3,0)}(S).
\end{equation}
By using the BWB Theorem \ref{thm-BWB-Gr}, it is readily seen that only the middle summand contributes to the cohomology of the bundle, and we have
\begin{equation}
\mH^3(\Omega\otimes \Omega^2)\cong L_0^{\oplus 2}, \quad \quad\mH^i(\Omega\otimes \Omega^2)=0 ~(i\neq 3).
\end{equation}

Thus the exact sequence \eqref{eqn-Gr24-les-2} becomes
\[
0\lra \mH^1(\mathcal{F}_2)\lra L_0^{\oplus 2}\lra   L_0^{\oplus 2} \lra \mH^2(\mathcal{F}_2)\lra 0,
\]
so that $\mH^1(\mathcal{F}_2)$ can contain at most two copies of the trivial representation of $SL_4$. On the other hand, it will have at least one copy of the trivial representation, since the map
\[
\tau\wedge(-): \mH^1(\Nt_P, T\Nt_P)^{-2}(\cong L_0)\lra \mH^{1}(\Nt_P,\wedge^3T\Nt_P)^{-4}
\]
is injective. 

To conclude the computation, we will resort to finding the relative Lie algebra cohomology of the sheaf $\mathcal{F}_2$. Lemma \ref{lem-reduce-to-flag} allows us to pull this bundle back to the full flag variety of $SL(4)$, and compute the sheaf cohomology for the bundle 
\[
G \times _B V_3^{-4} = G\times_B \left(\frac{\g\otimes \n_P\wedge\n_P\oplus \lu_P\otimes \n_P\wedge \n_P\wedge \n_P}{\Delta(\p)\wedge \n_P\wedge \n_P}\right).
\]

Then Theorem \ref{Bott_rel_lie}  and Lemma \ref{rel_lie_BGG} states that  
\[ \op{Hom}_G(L_0, \mH^1(G/B, G \times _B V_3^{-4}) \cong \op{H}^1(\lb, \h, V_3^{-4})  \] 
can be computed as the cohomology of the complex (\ref{BGG_complex}) for $E = V_3^{-4}$. 

To compute this relative Lie algebra cohomology, we need to find the dominant  weights $\lambda$ such that  $s_1 \cdot \lambda, \; s_2 \cdot \lambda$ or $s_3 \cdot \lambda$ is a weight  in $V_3^{-4}$.   
The only weight with this property is $\lambda =0$. There are no elements of zero weight in $V_3^{-4}$, and the only nontrivial weight space with weights of the form $w \cdot 0$ and $l(w) =1$ is for $w=s_2$. This weight space is $4$-dimensional and it is spanned by the elements  
\[ \{ e_{12} \otimes f_2 \wedge f_{12}, \;\; e_{23} \otimes f_2 \wedge f_{23} , \;\;  e_{123} \otimes f_2 \wedge f_{123} , \;\; e_{123} \otimes f_{12} \wedge f_{23} \}  \] 
where we denote 
$$e_{12} = [e_2, e_1], \quad e_{23} = [e_3, e_2], \quad e_{123} = [e_3, [e_2, e_1]],$$ 
$$f_{12} = [f_1, f_2], \quad f_{23} = [f_2, f_3], \quad f_{123} = [f_1, [f_2, f_3]] = [[f_1, f_2], f_3]$$ 
in terms of the standard Chevalley generators. 

 Then the complex (\ref{BGG_complex}) 
 for $E = V_3^{-4}$ and $\lambda =0$ becomes 
 \begin{equation} \label{BGG-V_34}
 0=V_3^{-4}[0] \stackrel{d_1^*}{\lra} V_3^{-4} [s_2 \cdot 0]   
 \stackrel{d_2^*}{\lra} \bigoplus_{l(w)=2} V_3^{-4}[w \cdot 0]   \lra \cdots  \ .
 \end{equation} 
 
 \begin{lemma} The first cohomology group of the complex (\ref{BGG-V_34}) is $2$-dimensional. 
 \end{lemma} 
 \begin{proof} 
 The image of the map $d_1^*$ is zero. The map $d_2^*$ acting on a subspace of weight $s_2 \cdot 0$ 
 is given explicitly in the table in \cite[Section 4.3]{LQ1}: 
\[ \left. d_2^{*}\right|_{s_2 \cdot 0} =   -f_1^2 \oplus (2 f_2 f_1 - f_1 f_2 ) \oplus f_3^2 \oplus (-2 f_2 f_3 + f_3 f_2) .\] 

Suppose that $v= a e_{12} \otimes f_2 \wedge f_{12} +b  e_{23} \otimes f_2 \wedge f_{23} +c  e_{123} \otimes f_2 \wedge f_{123} +d e_{123} \otimes f_{12} \wedge f_{23}$ is an element in $V_3^{-4}[s_2 \cdot 0]$. Then we compute 
\[ f_1^2  (v) = (2b-2c-2d) e_{23} \otimes f_{12} \wedge f_{123} \] 
\[ f_3^2 (v) = (2a -2c+2d) e_{12} \otimes f_{23} \wedge f_{123} , \] 
so for $v$ to be in the kernel of $d_2^*$, we need to set $a = c-d$ and $b=c+d$. 
Acting on this $v$ by the remaining two maps, we obtain 
\[ (2f_2 f_1 - f_1 f_2) (v) =  (c-d) u_1,  \quad  \quad (2f_2 f_3 - f_3 f_2) (v) = (c+d) u_2, \]
where 
\[ u_1 =  2 h_2 \otimes f_2 \wedge f_{12} + h_1 \otimes f_2 \wedge f_{12}- e_3 \otimes f_{12} \wedge f_{23} + e_3 \otimes f_2 \wedge f_{123} , \] 
and 
\[ u_2 = -2 h_2 \otimes f_2 \wedge f_{23} - h_3 \otimes f_2 \wedge f_{23}  + e_1 \otimes f_2 \wedge f_{123}   
+ e_1 \otimes f_{12} \wedge f_{23} . \] 

Both expressions $u_1$ and $u_2$ lie in the submodule $\Delta(\p)\wedge \n_P\wedge \n_P$. 
For example, we have 
\[ 
\begin{array}{rcl}  
\Delta (h_1 \otimes f_2 \wedge f_{12}) & = & h_1 \otimes f_2  \wedge f_{12}   + e_{23}  \otimes f_{23} \wedge f_2 \wedge f_{12} - e_{123} \otimes f_{123} \wedge f_2 \wedge f_{12} , \\  
\Delta (h_2 \otimes f_2 \wedge f_{12}) & = & h_2 \otimes f_2 \wedge f_{12} - e_{23} \otimes f_{23} \wedge f_2 \wedge f_{12} , \\ 
\Delta (e_3 \otimes f_{12} \wedge f_{23}) & = & e_3 \otimes f_{12} \wedge f_{23} - e_{23} \otimes f_2 \wedge f_{12} \wedge f_{23} , \\ 
\Delta ( e_3 \otimes f_2 \wedge f_{123} ) & = & e_3 \otimes f_2 \wedge f_{123} - e_{123} \otimes f_{12} \wedge f_2 \wedge f_{123} . \\
\end{array}  \] 
This implies the equality 
\[ u_1 \equiv 0 \quad (\op{mod}~\Delta (\p) \wedge \n_P \wedge \n_P ) .\] 
The computation for $u_2$ is similar. 

Thus we obtain that the kernel of $d_2^*$ in the complex (\ref{BGG-V_34}) is $2$-dimensional, and so is the first cohomology group of this complex.  
\end{proof}  

Therefore the cohomology $\op{H}^1(\lb, \h, V_3^{-4})$ is $2$-dimensional and we have  $\mH^1(\F_2) \cong L_0^{\oplus 2}$. 

Finally, we have obtained the dimension of the remaining components of the block of the center 
\[ 
h^1 (\wedge^3 T \Nt_P) = h^1( \wedge^5 T \Nt_P) = 2 .
\] 
This concludes the validation of the formal Hodge diamond (Table \ref{table-hodge-gr24}) for $Gr(2,4)$ .

%
%
%

\subsection{Regular blocks for type \texorpdfstring{$B_2$}{B2}}\label{app-b2}

In the second appendix, let us fix $\g=\mathfrak{so}_5$. We show that, outside of type $A$, the trivial module conjecture formulated in \cite[Conjecture 4.7]{LQ1} fails even in the simplest case of $B_2$. Therefore, the conjecture is purely a type-$A$ phenomenon.

To do so, we show that the formal Hodge diamond (Definition \ref{def-hodge-diamond}) for a regular block center of $\ul_q(\mathfrak{so}_5)$ has the following bigraded dimension table.

\begin{gather}  \label{table-hodge-b2} 
\begin{array}{|c||c|c|c|c|c|} \hline 
         \scriptstyle{ j+i=0 } & 1   &                         &                         &            &              \\  \hline 
              \scriptstyle{ j+i=2 } & 2   &   1                 &                         &           &               \\  \hline 
               \scriptstyle{ j+i =4 }   & 2   &  2   &         1 &        &                  \\  \hline 
               \scriptstyle{j + i = 6} &  2  & 3 + 5  & 2    & 1  &      \\ \hline 
             \scriptstyle{ j+i=8 }    &  1   &  2  &  2  &   2  & 1   \\  \hline \hline 
 \scriptstyle{h^{i,j}} &  \scriptstyle{ j-i=0 }       &     \scriptstyle{ j-i=2 }       &   \scriptstyle{ j-i=4 } &  \scriptstyle{ j-i=6 }   &  \scriptstyle{ j-i=8 }   \\ \hline 
\end{array}  
\end{gather}

Here in this table, we have written the entry $h^{2,4}$ as a sum
$$h^{2,4}=\mathrm{dim}(\mH^2(\Nt,\wedge^4T\Nt)^{-6})=8=3+5.$$
This is because, as we will see that, in fact, there is a decomposition
\begin{equation}\label{eqn-B2-interesting-term}
\mH^2(\Nt,\wedge^4T\Nt)^{-6}\cong L_0^{\oplus 3}\oplus L_{\alpha_1+\alpha_2}
\end{equation}
of $SO_5$ representations. Here $L_0$ stands for the trivial representation for $SO_5$, while $L_{\alpha_1+\alpha_2}$ is the simple module with highest weight the sum of two simple positive roots $\alpha_1$, $\alpha_2$ for $\mathfrak{so}_5$, i.e., it is the $5$-dimensional defining representation for $SO_5$. Other than this entry, the rest of table consist of direct sums of trivial $SO_5$ representations.

To establish Table \ref{table-hodge-b2}, we fix some notation in this situation. Consider the Chevalley generators
$
e_1, e_2, h_1, h_2, f_1, f_2\in \mathfrak{so}_5
$
subject to the relations
\[
[h_1,e_1]=2e_1, \quad [h_1,e_2]=-e_2, \quad [h_2, e_1]=-2e_1, \quad [h_2,e_2]=2e_2,
\]
\[
[h_1,f_1]=-2f_1, \quad [h_1,f_2]=f_2, \quad [h_2, f_1]=2f_1, \quad [h_2,f_2]=-2f_2,
\]
and the Serre relations
\[
e_2^3e_1-3e_2^2e_1e_2+3e_2e_1e_2^2-e_1e_2^3=0,\quad \quad  e_1^2e_2-2e_1e_2e_1+e_2e_1^2=0, 
\]
\[
f_2^3f_1-3f_2^2f_1f_2+3f_2f_1f_2^2-f_1f_2^3=0,\quad \quad f_1^2f_2-2f_1f_2f_1+f_2f_1^2=0.
\]

The weight of $e_1$ (resp.~$f_1$)is the long root $\alpha_1$ (resp.~$-\alpha_1$) while the weight of $e_2$ (resp.~$f_2$) is the short root $\alpha_2$ (resp.~$-\alpha_2$). For latter use, let us also fix
\[
f_3:=[f_1,f_2],\quad f_4:=[f_3,f_2], \quad e_3:=[e_2,e_1],\quad e_4:=[e_2,e_3].
\]

Proposition \ref{BGG_complex} in this case translates into the following particular complex. For any $\lb$-module $E$, the multiplicity space
\[
\mathrm{Hom}_{SO_5}(L_0,\mH^i (X, G\times_B E))
\]
is computed as the dimension of the $i$th cohomology group of the complex 
\begin{equation} \label{eqn-BGG-B2}
\begin{gathered}
\xymatrix@C=0.75em{
&& E[-\alpha_1]\ar[rr]^-{f_2^3}\ar[ddrr]^(0.25){u_1} && E[-\alpha_1-3\alpha_2]\ar[rr]^-{-f_1^2}\ar[ddrr]^(0.25){v_1}&& E[-3\alpha_1-3\alpha_2] \ar[drr]^-{-f_2} && \\
E[0] \ar[urr]^{f_1}\ar[drr]_{-f_2} && && && && E[-4\alpha_1-4\alpha_1]\\
&& E[-\alpha_2]\ar[rr]_-{f_1^2}\ar[uurr]_(0.25){u_2} && E[-2\alpha_1-\alpha_2]\ar[rr]_-{f_2^3}\ar[uurr]_(0.25){v_2} && E[-2\alpha_1-4\alpha_2]\ar[urr]_-{f_1} &&
} 
\end{gathered}\ ,
\end{equation}
where the middle arrows are given by the action of the elements
\[
u_1= 2f_1f_2-f_2f_1, \quad \quad u_2= 3f_2^2f_1-3f_2f_1f_2+f_1f_2^2,
\]
\[
v_1= f_1f_2-2f_2f_1, \quad \quad v_2= f_2^2f_1-3f_2f_1f_2+3f_1f_2^2.
\]

We will perform the sample computations that lead to the decomposition \eqref{eqn-B2-interesting-term}, and leave the rest of the computations for Table \ref{table-hodge-b2} to the reader as exercises.

On the flag variety $X=G/B$, we have an identification of the bundle
\begin{equation}\label{eqn-b2-B-mod-structure}
\mathrm{pr}_*(\wedge^4 T\Nt)^{-6}\cong G\times_B\frac{\g\otimes \wedge^3\n\oplus \lu\otimes \wedge^4 \n}{\Delta(\lb\otimes \wedge^3 \n)} ,
\end{equation}
where $\Delta$ is induced from the diagonal map \eqref{eqn-Delta-map}. It follows that we have a short exact sequence of vector bundles on $X$:
\begin{equation}\label{eqn-b2-ses-bundles}
0\lra T_X\otimes \Omega_X^4 \lra \mathrm{pr}_*(\wedge^4 T\Nt)^{-6} \lra T_X\otimes \Omega_X^3 \lra 0.
\end{equation}
We will compute the cohomology groups $\mH^2(X,T_X\otimes \Omega_X^4)$ and $\mH(X,T_X\otimes \Omega_X^3)$ as an intermediate step.

\begin{lemma} \label{lem-auxiliary-bundle-b2}
For the flag variety $X$ of $\mathfrak{so}_5$, there are isomorphisms of $\mathfrak{so}_5$-representations
\[
\mH^i(X,T_X\otimes \Omega_X^4)
\cong 
\left\{
\begin{array}{cc}
L_0^{\oplus 2} & i=3,\\
0 & i\neq 3.
\end{array}
\right.
\quad \quad
\mH^i(X,T_X\otimes \Omega_X^3)
\cong
\left\{
\begin{array}{cc}
 L_0^{\oplus 3}
 \oplus L_{\alpha_1+\alpha_2} & i=2,\\
 0 & i\neq 2.
\end{array}
\right.
\]
\end{lemma}
\begin{proof}
In the course of the proof, we will drop $X$ indexing the flag variety of various bundles. Since $X$ is $4$-dimensional, $\Omega^4\cong K_X$ is the canonical bundle, and Serre duality tells us that
\[
\mH^\bullet (X, T\otimes \Omega^4)\cong \mH^{4-\bullet}(X, \Omega)^*.
\]
The first vanishing result follows since $\mH^\bullet(X,\Omega)$ is only nonzero in cohomological degree $1$.

Likewise, one has
\[
\mH^i(X,T\otimes \Omega^3)\cong \mH^i(X,\Omega\otimes \Omega)^*\cong \mH^i(X,S^2(\Omega))\oplus \mH^i(X, \Omega^2)^*.
\]
As the total cohomology of $\mH^i(X,\Omega^2)$ is known to be two copies of $L_0$ concentrated in cohomological degree $2$, we are reduced to computing $\mH^i(X,S^2(\Omega))$, where $S^2(\Omega)$ stands for the second symmetric product of the bundle $\Omega$.

To illustrate the general method, we will compute the relative Lie algebra cohomology of $S^2(\n)$. An inspection on the weights shows that the weight spaces of $S^2(\n_P)$ are all singular except in the following cases
\begin{enumerate}
\item[(1)] weight $-2\alpha_1$, spanned by $f_1^2$,
\item[(2)] weight $-2\alpha_1-\alpha_2$, spanned by $f_1f_3$,
\item[(3)] weight $-\alpha_1-3\alpha_2$, spanned by $f_2f_4$,
\item[(4)] weight $-2\alpha_1-2\alpha_2$, spanned by $f_4^2$.
\end{enumerate}
It is easily checked that $f_1^2$ is the only weight that lies in the orbit of the $(-\rho)$-shifted Weyl group action on $\alpha_1+\alpha_2$, and the unique element in the Weyl group that brings it to the dominant chamber has length two. It follows $L_{\alpha_1+\alpha_2}$ will figure in the cohomology of $S^2(\Omega)$ with multiplicity one in degree  $2$. The remaining three weights in (2), (3) and (4) all lie on the orbit of the $0$ weight, and the weight spaces are assembled into one BGG complex of \eqref{eqn-BGG-B2}:
\[
0\lra \C f_2f_4\oplus \C f_1f_3\xrightarrow{(f_1f_2-2f_2f_1, f_2^3)} \C f_4^2\lra 0
\]
where the first nontrivial term sits in cohomological degree $2$. It is then readily computed that the cohomology of this complex is $1$-dimensional in cohmological degree $2$.
\end{proof}

Taking cohomology of the short exact sequence \eqref{eqn-b2-ses-bundles}, we obtain a long exact sequence
\begin{equation}
0\lra \mH^2 (\Nt,\wedge^4 T\Nt)^{-6} \lra \mH^2(X,T_X\otimes \Omega_X^3)\lra \mH^3(X,T_X\otimes\Omega_X^4)\lra \cdots.
\end{equation}
Using Lemma \ref{lem-auxiliary-bundle-b2}, we are reduced to
\begin{equation}\label{eqn-H2-sequence}
0\lra \mH^2 (\Nt,\wedge^4 T\Nt)^{-6} \lra
L_0^{\oplus 3}\oplus L_{\alpha_1+\alpha_2} \lra L_0^{\oplus 2} \lra \mH^3 (\Nt,\wedge^4 T\Nt)^{-6} \lra 0.
\end{equation}
It follows that $\mH^2 (\Nt,\wedge^4 T\Nt)^{-6}$ has a copy of $L_{\alpha_1+\alpha_2}$ as an $\mathfrak{so}_5$-subrepresentation, and furthermore, it will contain at most $3$, and at least $2$ (using the $\mathfrak{sl}_2$ action on the formal Hodge diamond), copies of $L_0$ as summands.

\begin{lemma}
There is an isomorphism of $\mathfrak{so}_5$ representations
\[
\mH^i (\Nt,\wedge^4 T\Nt)^{-6}
\cong
\left\{
\begin{array}{cc}
L_0^{\oplus 3} \oplus L_{\alpha_1+\alpha_2} & i=2,\\
L_0^{\oplus 2} & i=3,\\
0 & i\neq 2,3.
\end{array}
\right.
\]
\end{lemma}
\begin{proof}
By the discussion above, we are reduced to computing the multiplicity of $L_0$ in the Hochschild cohomology group $\mH^i(\Nt,\wedge^2T\Nt)^{-6}$. We resort to the BGG complex \eqref{eqn-BGG-B2} together with the explicit equivariant structure \eqref{eqn-b2-B-mod-structure} of the bundle for this computation. In this situation, the BGG complex reduces to
\[
0 \lra 
\begin{array}{c}
\C e_4\otimes f_1\wedge f_2 \wedge f_3
\end{array}
\lra
\begin{array}{c}
\C e_1\otimes f_1\wedge f_2 \wedge f_4\\
\oplus\\
\C e_3\otimes f_2\wedge f_3 \wedge f_4\\
\oplus\\
\C e_2\otimes f_1\wedge f_2 \wedge f_3\\
\oplus\\
\C e_4\otimes f_1\wedge f_3 \wedge f_4
\end{array} 
\lra 
\begin{array}{c}
\C e_1\otimes f_1\wedge f_2 \wedge f_3 \wedge f_4\\
\oplus\\
\C e_2\otimes f_1\wedge f_2 \wedge f_3 \wedge f_4\\
\end{array}
\lra 0,
\]
where the left most nonzero term sits in homological degree one, and the differentials $d_1^*$, $d_2^*$ are given in \eqref{eqn-BGG-B2}. It is easily checked that $d_1^*$ is injective, and $d_2^*$ sends the middle terms to zero modulo $\Delta(\lb\otimes \wedge^3 \n)$. The details will be left to the reader.
\end{proof}


\addcontentsline{toc}{section}{References}

\bibliographystyle{alpha}
\bibliography{qy-bib}


%

\vspace{0.1in}

\noindent A.~L.: { \sl \small  \'{E}cole Polytechnique F\'{e}d\'{e}rale de Lausanne,  SB MATH TAN, MA C3 535 (B\^atiment MA), Station 8, CH-1015 Lausanne, Switzerland} \newline \noindent {\tt \small email: anna.lachowska@epfl.ch}

\vspace{0.1in}

\noindent Y.~Q.: { \sl \small Department of Mathematics, Yale University, New Haven, CT 06511, USA} \newline \noindent {\tt \small email: you.qi@yale.edu}

%

\end{document}